\documentclass[preprint,12pt]{elsarticle}
\usepackage{graphicx}
\usepackage{subfigure}
\usepackage{amsmath, stmaryrd}
\usepackage{mathtools, amsfonts, amssymb, amsthm}%SIAM模板修正
\usepackage{enumerate, url}
\usepackage[margin=1in]{geometry}
\usepackage[normalem]{ulem}
\usepackage{caption}
\usepackage{algorithm}
\usepackage{algpseudocode}
\usepackage{mleftright}
\usepackage{tikz-cd}
\usepackage[normalem]{ulem}
\usepackage{hyperref}
\usepackage{appendix}
\usepackage{pgf}
\usepackage{kbordermatrix}
\usepackage{tikz}
\usetikzlibrary{arrows,automata}
\usepackage{booktabs}

\usepackage{mathtools}
\usepackage{color}
\usepackage{graphicx}
\usepackage{pgfplots}
\usetikzlibrary{shapes.geometric}

\usepackage{xcolor}
\hypersetup{
    colorlinks,
    linkcolor={red!50!black},
    citecolor={blue!50!black},
    urlcolor={blue!80!black}
}

\newtheorem{theorem}{Theorem}[section]
\newtheorem{lemma}[theorem]{Lemma}
\newtheorem{corollary}[theorem]{Corollary}
\newtheorem{proposition}[theorem]{Proposition}
\theoremstyle{definition}
\newtheorem{definition}[theorem]{Definition}
\newtheorem{example}[theorem]{Example}
\newtheorem{remark}[theorem]{Remark}

\newcommand{\tp}{{\scriptscriptstyle\mathsf{T}}}

\let\rank\undefined%模板修正
\DeclareMathOperator{\rank}{rank}
\DeclareFontFamily{U} {MnSymbolC}{}
\DeclareFontShape{U}{MnSymbolC}{m}{n}{
  <-6> MnSymbolC5
  <6-7> MnSymbolC6
  <7-8> MnSymbolC7
  <8-9> MnSymbolC8
  <9-10> MnSymbolC9
  <10-12> MnSymbolC10
  <12-> MnSymbolC12}{}
\DeclareFontShape{U}{MnSymbolC}{b}{n}{
  <-6> MnSymbolC-Bold5
  <6-7> MnSymbolC-Bold6
  <7-8> MnSymbolC-Bold7
  <8-9> MnSymbolC-Bold8
  <9-10> MnSymbolC-Bold9
  <10-12> MnSymbolC-Bold10
  <12-> MnSymbolC-Bold12}{}

\DeclareSymbolFont{MnSyC} {U} {MnSymbolC}{m}{n}

\DeclareMathSymbol{\plus}{\mathrel}{MnSyC}{20}

\newcommand{\red}[1]{{#1}}
\newcommand{\cyan}[1]{{#1}}

\newcommand{\Rmnum}[1]{\expandafter\@slowromancap\Romannumeral #1@}

\newcommand{\overlongrightarrow}[2]
{\vbox{\offinterlineskip
      \halign{##\cr
              \hfill #1 \hfill\cr
              \hbox to #2{\relax\rightarrowfill}\cr}}\ignorespaces
}
\newcommand{\overlongleftarrow}[2]
{\vbox{\offinterlineskip
      \halign{##\cr
              \hfill #1 \hfill\cr
              \hbox to #2{\relax\leftarrowfill}\cr}}\ignorespaces
}
\usepackage{amssymb}
\journal{Applied and Computational Harmonic Analysis}

\begin{document}

\begin{frontmatter}



\title{ Computing sparse Fourier sum of squares on finite abelian groups in quasi-linear time\thanks{\quad  Ke Ye and  Lihong  Zhi are  supported by the National Key Research Project of China 2018YFA0306702. Lihong Zhi is supported by  the National Natural Science Foundation of China 12071467.}}

 \author[label1]{Jianting Yang}
  \author[label2]{Ke Ye}
    \author[label2]{Lihong Zhi}
 \affiliation[label1]{organization={CNRS@CREATE},
             addressline={1 CREATE Way,},
             city={Singapore},
             postcode={138602},
             state={Singapore},
             country={Singapore}}

 \affiliation[label2]{organization={ Key Lab of Mathematics Mechanization, AMSS, University of Chinese Academy of Sciences},
             addressline={No.55 Zhongguancun East Road},
             city={Beijing},
             postcode={100190},
             state={Beijing},
             country={China}}

\begin{abstract}
The problem of verifying the nonnegativity of a function on a finite abelian group is a long-standing challenging problem.  %{Related applications can be found in various fields such as delegated computation, sum of Hermitian squares (SOHS) and combinatorial optimization.}
%Let $X$ be a finite set and let $F$ be a nonnegative real-valued function on $X$. This  paper concerns the sum of squares (SOS) sparsity of $F$ in the sense of Fourier support \cite{fawzi2016sparse}. More precisely, if we equip $X$ with an abelian group structure, i.e., we choose a finite abelian group $G$ together with a bijection $\varphi:G \mapsto X$, then $f\coloneqq F\circ \varphi$ can be naturally identified with an element in the group algebra $\mathbb{C}[G]$.
 The basic theory of representation theory of finite groups indicates that a function $f$ on a finite abelian group $G$ can be written as a linear combination of characters of irreducible representations of $G$ by $ f(x)=\sum_{\chi \in \widehat{G}}  \widehat{f} (\chi)\chi(x)$,
where $\widehat{G}$ is the dual group of $G$ consisting of all characters of $G$ and $ \widehat{f} (\chi)$ is the \emph{Fourier coefficient} of $f$ at $\chi \in \widehat{G}$. In this paper, we show that by performing the fast (inverse) Fourier transform, we are able to compute a sparse Fourier sum of squares (FSOS) certificate of $f$ on a  finite abelian group  $G$ with complexity \if \cyan{$\operatorname{O}\left(|G| \log(|G|)+\log(k_{\min})\operatorname{SDP}(2k_{\min})\right)$},\fi that is quasi-linear in the order of $G$ and polynomial in the FSOS sparsity \if $k_{\min}$\fi of $f$. Moreover, for a nonnegatvie function $f$ on a finite abelian group $G$ \red{and} a set $S \subset \widehat{G}$, we give a lower bound of the constant $M$ such that  $f+M$ admits an FSOS \red{supported on} $S$. We demonstrate the efficiency of the proposed algorithm by numerical experiments on various abelian groups of orders up to $10^7$. \red{As applications, we also solve some combinatorial optimization problems and the sum of Hermitian squares (SOHS) problem \if on $\mathbb{T}^n$\fi by \cyan{sparse} FSOS.}
% Related applications  have been given for solving combinatorial optimization problems and  computing sum of Hermitian squares (SOHS) on $\mathbb{T}^n$.
% (or with FSOS sparsity equals to $\operatorname{O}(s)$).}% Furthermore, our algorithm minimizes the given lower bound.
%It is also noticeable that different choices of group structures on $X$ would result in different FSOS sparsities of $F$. Along this line, we investigate upper bounds for FSOS sparsities with respect to different choices of group structures, which generalize and refine existing results in the literature. More precisely, \emph{(i)} we give an upper bound for FSOS sparsities of nonnegative functions on the product and the quotient of two finite abelian groups respectively; \emph{(ii)} we prove the equivalence between finding the group structure for the Fourier-sparsest representation of $F$ and solving an integer linear programming problem. {We also provide several examples in delegated computation, sum of Hermitian squares (SOHS) and combinatorial optimization, to show the advantage of FSOS.}}
\end{abstract}




\end{frontmatter}



\section{Introduction}
Let $X$ be a finite set and let $F$ be a nonnegative real-valued function on $X$. This  paper concerns the sum of squares (SOS) sparsity of $F$ in the sense of Fourier support \cite{fawzi2016sparse}. More precisely, if we equip $X$ with an abelian group structure, i.e., we choose a finite abelian group $G$ together with a bijection $\varphi:G \mapsto X$, then $f\coloneqq F\circ \varphi$ can be naturally identified with an element in the group algebra $\mathbb{C}[G]$. The basic theory \cite[Chapter~1]{fulton2013representation} of representation theory of finite groups indicates that the function $f$ can be written as a linear combination of \emph{characters} of irreducible representations of $G$:
\begin{equation*}
  f(x)=\sum_{\chi \in \widehat{G}}  \widehat{f} (\chi)\chi(x), \quad  x \in G,
\end{equation*}
where $\widehat{G}$ is the dual group of $G$ consisting of all characters of $G$ and $ \widehat{f} (\chi)$ is the \emph{Fourier coefficient} of $f$ at $\chi \in \widehat{G}$. A subset $S\subseteq \widehat{G}$ containing those $\chi \in \widehat{G}$ such that $\widehat{f}(\chi) \neq 0$ is called a \emph{Fourier support} of $f$. An Fourier sum of squares (FSOS) certificate of the nonnegative function $f$ on finite abelian group $G$ is in form of 
\[ f=\sum_{i\in I} |g_i|^2, \]
where $\{g_i\}_{i\in I}$ is a finite family  of functions  on group $G$. Clearly, an  FSOS certificate of $f$ indeed proves that $f$ is a nonnegative function on $G$. With the aim of reducing computational complexity, this paper focuses on computing a sparse FSOS certificate.
%According to \cite[Theorem~1]{fawzi2016sparse}, $f$ has a {Fourier sum of squares (FSOS)} certificate with support ${T} \subseteq \widehat{G}$, where ${T}$ is a Fourier support of a chordal cover of the Cayley graph $\operatorname{Cay}(\widehat{G},{S})$.We notice that $T$ is the set of characters  appeared in the {FSOS} certificate of $f$ and it is different from the Fourier support $S$ of $f$. The cardinality $|T|$ of $T$ is called the \emph{FSOS sparsity of $F$ with respect to the pair $(G,\varphi)$}. When $G$ and $\varphi$ are fixed, we simply call $|T|$ the \emph{FSOS sparsity of $f$}.

\if
Obviously, different choices of bijections $\varphi: G\to X$ would result in different FSOS sparsity of $F$. In fact, this simple observation is the impetus of this work. Our goal is twofold:
\begin{enumerate}[(i)]
\item to propose a quasi-linear algorithm for sparse FSOS;
\item to provide a better upper bound for the FSOS sparsity.
\end{enumerate}
On the one hand, from a practical point of view, there is no doubt that an algorithm for sparse FSOS certificate is of great importance, just as its counterpart in polynomial optimization  \cite{WKKM06,WLT18,ZFP19,wang2021chordal}. However, as far as we know, no such algorithm, let alone an efficient one, can be found in existing works. In Section~\ref{sec:algorithms}, we present Algorithm~\ref{alg:1} for finding a universal sparse FSOS certificate of a nonnegative function on a finite abelian group. Here a universal sparse FSOS refers to a sparse FSOS which is independent of specific values of Fourier coefficients of the function. Unfortunately, Theorem~\ref{thm:support=feasible} indicates that the problem of finding a universal sparse FSOS certificate is equivalent to solving the integer linear programming problem \eqref{ILP}. As it is well-known, an integer linear programming problem is in general NP-complete \cite{schrijver1998theory}. Thus it is unrealistic to expect Algorithm~\ref{alg:1} to be efficient. Instead, we propose Algorithm~\ref{alg:4} for a sparse FSOS, which has quasi-linear complexity since it takes specific values of Fourier coefficients of the function into account.

On the other hand, it is also imperative to theoretically bound the FSOS sparsity of a given nonnegative function on a finite abelian group. In the literature, such an upper bound is only known for  two extreme cases: $\mathbb{Z}_2^n$ and $\mathbb{Z}_N$. We refer interested readers to \cite{sakaue2017exact} and \cite{fawzi2016sparse} for  more details. In fact, even in these two cases, upper bounds of the FSOS sparsity can be extremely different, as we will see in the following example: Suppose that $N = 2^n$ and $n = 2^k$ for $k\ge 1$. Let $\mathcal{F} $ be the function space $\{f:\mathbb{Z}_2^n \rightarrow \mathbb{C},~\deg(f) \leq r=2 \}$ whose dimension is
\[
\dim \mathcal{F} = \binom{n}{0} + \binom{n}{1} + \binom{n}{2}  = 1 + 2^{k-1}(2^k + 1) \le 2^{2k}.
\]
By \cite[Theorem 3.2]{sakaue2017exact}, the degree bound of the FSOS certificate of a nonnegative function $f\in \mathcal{F} $ is
\[
s = \lceil (n+r - 1)/2 \rceil = 2^{k-1} + 1,
\]
which implies that the FSOS sparsity  of $f$ (as a function on $\mathbb{Z}_2^n$) is at most
\[
c(\mathbb{Z}^n_2) \coloneqq \sum_{i=0}^{2^{k-1}+ 1} \binom{n}{i}.
\]
Let $\mathcal{F}'$ be the function space $\{g:\mathbb{Z}_N \rightarrow \mathbb{C},~ \deg(g) \leq d = 2^{2k-1} \}$. Then \cite[Theorem 3]{fawzi2016sparse} implies that the FSOS sparsity of a nonnegative function $g\in \mathcal{F}'$ is at most
\[
c(\mathbb{Z}_N) \coloneqq 3 d \log_2(N/d) = 3 (2^k - 2k + 1) 2^{2k-1}.
\]
When $k
\ge 4$, we obviously have $c(\mathbb{Z}^n_2)  \ge c(\mathbb{Z}_N) $. In this vein, we prove in Section~\ref{sec:upper bound} the following two theorems, which provide an upper bound for the FSOS sparsity when $G$ is a product or a quotient of two finite abelian groups.
\begin{theorem}\label{Th1}
Let  $f:G_1 \times G_2 \rightarrow \mathbb{R}$ be a nonnegative function on a finite abelian group $G_1 \times G_2$, with support ${S}_1 \times {S}_2 \subseteq \widehat{G_1 \times G_2}$, where $  {S}_1 \subseteq \widehat{G}_1 $ and $  {S}_2 \subseteq \widehat{G}_2 $. If $\Gamma_1$ is a chordal cover of the Cayley graph $\operatorname{Cay}(\widehat{G}_1, {S}_1)$ with a Fourier support ${T}_1$ and ${S}_2$ is the generating set of the subgroup $\widehat{H} \subseteq \widehat{G}_2$, then $f$ has an SOS certificate with a Fourier support \[{T}={T}_1 \times \widehat{H}.\] In particular, we have $|{T}| \leq |{T}_1| \cdot |\widehat{H}|$.
\end{theorem}

Let $G$ be a finite abelian group and let $H$ be a subgroup of $G$.
 We denote by $\chi_0$ the trivial character, i.e., $\chi_0(x)=1$ for all $ x \in G$. We observe that the dual group $\widehat{H}$ of $H$ is a subgroup of $\widehat{G}$.
\begin{theorem}\label{Th2}
 Assume that $\widehat{H}$ is generated by ${S} \subseteq \widehat{G}$. If $\Gamma$ is a chordal cover of $\operatorname{Cay}(\widehat{H},{S})$ with Fourier support ${T}$, then $\Gamma \boxtimes \operatorname{Cay}(\widehat{G}/\widehat{H},\{\widehat{H}\})$ is a chordal cover of $\operatorname{Cay}({\widehat{G}},{S})$ with Fourier support ${T}$. Moreover, if $\Gamma$ is a chordal cover of $\operatorname{Cay}(\widehat{H},{S})$ with minimal number of edges, then  $\Gamma \boxtimes \operatorname{Cay}(\widehat{G}/\widehat{H},\{\widehat{H}\})$ is  a chordal cover of $\operatorname{Cay}({\widehat{G}},{S})$ with minimal number of edges.
\end{theorem}

Before we proceed, we would like to emphasize the importance of Theorems~\ref{Th1} and \ref{Th2}. Due to  the fundamental theorem of finite abelian groups, $\mathbb{Z}_N$ and $\mathbb{Z}_2^n$ are not the only possible abelian group structures one can impose on a finite set. For instance, there exist $3$ non-isomorphic abelian group structures on a set of cardinality $8$: $\mathbb{Z}_8$, $\mathbb{Z}_2 \times \mathbb{Z}_4$ and $\mathbb{Z}_2^3$. Moreover, Theorems~\ref{Th1} and \ref{Th2} provide us very useful tools to obtain an upper bound for the cardinality of a Fourier support of a nonnegative function on a finite abelian group. Even in the case of $\mathbb{Z}_2^n$, which is extensively studied in the literature \cite{amrollahi2019efficiently, fawzi2016sparse,kurpisz2019sum}, upper bounds obtained by Theorems~\ref{Th1} and \ref{Th2} could be  much better than previously known ones. As an example, we consider the function $f(x)=2-(x_1+x_2)\prod_{i=3}^n x_i$ on $\mathbb{Z}_2^n =\{-1,1\}^{n }$. It is clear that  $f$ is nonnegative. According to \cite[Theorem~3]{fawzi2016sparse}, the support of the sum-of-squares certificate of $f$ is contained in
 \[
 \{g \in \widehat{\mathbb{Z}}_2^{n}; \deg(g)\leq n-1\},
 \]
which has $2^n-1$ elements. However, if we regard ${\mathbb{Z}}_2^n$ as the product of  ${\mathbb{Z}}_2^{2}$ and ${\mathbb{Z}}_2^{n-2}$ and let ${S}_1=\{1,x_1,x_2\}$ and ${S}_2=\{1,\prod_{i=3}^n x_i\}$, then Theorem~\ref{Th1} implies that $f$ has an SOS certificate with only $6$  terms:
\[
\{1,x_1,x_2,\prod_{i=3}^n x_i,x_1\prod_{i=3}^n x_i,x_2\prod_{i=3}^n x_i\}.
\]
Furthermore, if we take  ${S}=\left\{1,x_1\prod_{i=3}^n x_i,x_2\prod_{i=3}^n x_i\right\}$ and $ \widehat{H}=\left\{1,x_1x_2,x_1\prod_{i=3}^n x_i,x_2\prod_{i=3}^n x_i\right\}$. According to Theorem \ref{Th2}, we may conclude that $f$ has an SOS certificate of cardinality three:
\[
\{1,x_1\prod_{i=3}^n x_i,x_2\prod_{i=3}^n x_i\}.
\]
To be more concrete, if $n = 4$, then the function $f(x_1,x_2,x_3,x_4) =2- x_1 x_3 x_4- x_2 x_3 x_4$ can be written as
\[
\frac{1}{4} \left( -{ x_1}\,{ x_3}\,{ x_4}+{ x_2}\,{ x_3}\,{ x_4
} \right) ^{2}+ \left(\frac{1}{2}{ x_1}\,{ x_3}\,{ x_4}
-1+\frac{1}{2}{ x_2}\,{ x_3}\,{ x_4} \right) ^{2},
\]
which gives a desired SOS certificate. Here the chordal cover of $\operatorname{Cay}(\widehat{H},{S})$ and $\operatorname{Cay}(\widehat{\mathbb{Z}}_2^4/\widehat{H},\{\widehat{H}\})$ are shown in Figure~\ref{fig:ex}.

\begin{figure}[h!]
  \centering
 \subfigure[Chordal cover of  $\operatorname{Cay}(\widehat{H},{S})$]{
 \begin{tikzpicture}[main/.style = {draw, circle,minimum width={45pt},}]
\node[main] (1) {$1$};
\node[main] (2) at	(	3	,	-3	)	{$x_1x_3x_4$};
\node[main] (3) at	(	0	,	-6	)	{$x_1x_2$};
\node[main] (4) at	(	-3	,	-3	)	{$x_2x_3x_4$};
\draw (1) -- (2);
\draw (2) -- (3);
\draw (3) -- (4);
\draw (2) -- (4);
\draw (1) -- (4);
\end{tikzpicture}
 }
 \quad
  \subfigure[$\operatorname{Cay}(\widehat{\mathbb{Z}}_2^4/\widehat{H},\{\widehat{H}\})$]{
  \begin{tikzpicture}[main/.style = {draw, circle,minimum width={45pt},}]
\node[main] (1) {$\widehat{H}$};
\node[main] (2) at	(	3	,	-3	)	{$x_3x_4\widehat{H}$};
\node[main] (3) at	(	0	,	-6	)	{$x_1x_4\widehat{H}$};
\node[main] (4) at	(	-3	,	-3	)	{$x_1x_3\widehat{H}$};
\end{tikzpicture}
  }
  \caption{Chordal cover of $\operatorname{Cay}(\widehat{H},{S})$ and $\operatorname{Cay}(\widehat{\mathbb{Z}}_2^4/\widehat{H},\{\widehat{H}\})$}
  \label{fig:ex}
\end{figure}
\fi 

In \cite{fawzi2016sparse,sakaue2017exact}, based on graph theory, the authors provide interesting  theoretical bounds on the sparsity of FSOS. In this paper, we \red{focus on the computational aspect of FSOS}. \if The complexity of the algorithm is quasi-linear in the order of $G$ and polynomial in the FSOS sparsity  of $f$.\fi Our main contributions are as follows:
\begin{enumerate}
    \item \red{We first formulate the problem of computing sparse FSOS as the optimization problem \eqref{Min:L0GramMatrix}. Next we prove in Theorem~\ref{lem:square root}
    that the square root of $f$ provides a closed-form solution to a properly formulated convex relaxation of \eqref{Min:L0GramMatrix}. Based on that, we design Algorithm~\ref{alg:4} to compute a sparse FSOS of a nonnegative function in quasi-linear time. Numerical experiments are presented in Tables~\ref{table4.1}, ~\ref{table4.2},~\ref{table4.3} to demonstrate the efficiency of our algorithm. }

    \item  \red{In Theorem~\ref{lem:square root}, we only select terms of large magnitude in $\sqrt{f}$ to compute a sparse FSOS.} \if We propose to  select  only  terms  in $\sqrt{f}$ with large absolute value of Fourier coefficients for computing sparse FSOS.\fi We expound the reasons why this heuristic term selection strategy  works very well in practice in Theorem~\ref{thm:square-root_based_error}. Furthermore, we show in Proposition~\ref{prop:invariant-isomorphism} that the terms selected by our method remain unchanged under the group isomorphism.%while  the terms  selected   via other methods does not possess  this property (Proposition~\ref{prop:invariant-isomorphism}).
   
    \item \red{Applications of FSOS to combinatorial optimization problems and the sum of Hermitian squares (SOHS) problem are presented}. Remarkably, we show that one can prove the pigeon-hole principle by an FSOS certificate of sparsity $O(n^2)$ in Proposition~\ref{prop6.2}. \red{As a comparison}, any resolution refutation requires \red{exponentially many} inference steps to prove the pigeon-hole principle 
    \cite[Theorem 16, Corollary 18]{bonet2007resolution}. \red{In Theorem~\ref{Thm:lift}, we give sufficient conditions for the existence of a lifting of FSOS on a finte abelian group to SOHS on \cyan{unit circles in complex planes}. Moreover, Example~\ref{motzkin} indicates that such a lifting may provide a much simpler certificate for the nonnegativity of polynomials on cubes.}
\end{enumerate}
%Compared to traditional term  selection methods, our  square-root based term  selection method have the following advantages:
% \begin{itemize}
  % \item Our method is based on solving the convex relaxation problem of FSOS sparsity minimization problem (Theorem~\ref{lem:square root}). 
      %\item Our method minimizes the error bound stated in Theorem~\ref{thm:square-root_based_error}.
  % \item The basis chosen based on our method remains unchanged   under the group isomorphism, while  the basis selected   via %previous other methods does not possess  this property (Proposition~\ref{prop:invariant-isomorphism}).
% \end{itemize}}
\red{We remark that although the graph theoretic approach in \cite{fawzi2016sparse} can be turned into an algorithm to compute sparse FSOS, it is inefficient since it only uses the Fourier support of the given function. However, more information about FSOS sparsity can be acquired by exploring the structure of coefficients of that function and this observation eventually leads to our Algorithm~\ref{alg:4}.}

 The rest of the paper is organized as follows: In Section~\ref{sec:algorithms}, we %  prove that the convex relaxation of the problem of minimizing the FSOS sparsity of a given nonnegative function is meaningless (Lemma~\ref{lem:constant l1-norm}), however,  a  convex relaxation of an equivalent form of  FSOS sparsity minimization problem is meaningful. Moreover, 
establish a proper formulation of the convex relaxation of  FSOS sparsity minimization problem. %It is clear that $\sqrt(f)$ can be computed in quasi-linear time to the  cardinality of the 
  %  is just  and can be computed in quasi-linear time to the  cardinality of the group (Theorem~\ref{lem:square root}). 
In Section~\ref{subsec:numerics},  we give an algorithm  to  compute a sparse FSOS of a given nonnegative function. Numerical experiments are provided to demonstrate the  correctness and efficiency of our algorithm. 
In section~\ref{sec:square-root:bound}, we present an error analysis \red{to validate} the term selection strategy in the algorithm. 
 In Section~\ref{sec6}, \red{we discuss applications of FSOS to combinatorial optimization problems and the SOHS probelm}. 

%we present an applications of FSOS in  sum of Hermitian squares (SOHS). %We first present a short certificate for the nonnegativity of the Motzkin polynomial in square $[-2,2]\times [-2,2]$. 
%We present a  degree bound of the FSOS of function which cannot be lifted to $\mathbb{T}^n$. Particularly, we prove that  for any $n\in \mathbb{N}$  with $n>4$, the function 
%\[g: \mathbb{Z}_n \to \mathbb{C},~x \mapsto 2\cos(\frac{2\pi}{2n})-\exp^{\frac{2\pi i}{2n}}\chi_{1}(x)-\exp^{-\frac{2\pi i}{2n}}\chi_{-1}(x) \]
% is nonnegative on $\mathbb{Z}_n$,  but it has no FSOS with degree less than $\frac{n}{4}$ %(Corollary~\ref{coro:linear-function-with-no-low-degree-FSOS}).
%Lastly, we show that one can prove the pigeon-hole principle by a short FSOS certificate. In contrast, any resolution refutation requires exponentially many steps to prove the pigeon-hole principle~\cite{bonet2007resolution}.}

\section{Preliminaries}\label{sec:pre}
In this section,  we recall some basic definitions and results in representation theory \cite{rudin1962fourier,fulton2013representation} and  graph theory \cite{hammack2011handbook,kakimura2010direct,vandenberghe2015chordal}.

\subsection{group theory and representation theory}
Let $G$ be a finite abelian group and let $H\subseteq G$ be a subgroup of $G$. % For each $x\in G$, the set $xH \coloneqq \{xh; h \in H\}$ is called the \emph{coset} of $H$ in $G$ represented by $x$.
 A  nonzero complex valued function $\chi$ on $G$ is called a \emph{character} of $G$ if it satisfies:
\begin{equation*}
  \chi(xy)=\chi(x)\chi(y),  \quad x,y \in G.
  \end{equation*}
 The set $\widehat{G}$ of all characters of $G$ is called the \emph{dual group} of $G$. A subset $S\subseteq \widehat{G}$ is called symmetric if $\chi \in S$ implies $\chi^{-1} \in S$. It is straightforward to verify that $\widehat{G}$ is a finite abelian group, with the group operation given by pointwise multiplication. Since $G$ is a finite abelian group, all irreducible representations of $G$ are one dimensional. Hence we may identify $\widehat{G}$ with the set of all irreducible representations of $G$.

The fundamental theorem \cite{dummit2004abstract} of finite abelian groups implies that
\[
G \simeq \mathbb{Z}_{n_1} \times \cdots \times \mathbb{Z}_{n_k}.
\]
For any positive integer $d$, we also have
\[
\widehat{\mathbb{Z}_d}= \{\chi_l(x) \coloneqq \exp^{\frac{2i \pi l x}{d}}, l=0,\dots, d-1\}.
\]
Moreover, if $G_1,G_2$ are  finite abelian groups, then $\widehat{G_1 \times G_2} = \widehat{G_1} \times \widehat{G_2}$. Therefore we can regard each $\chi\in \widehat{G}$ as
\[
\chi (x_1,\dots, x_k) = \prod_{j=1}^k \exp^{\frac{2i\pi l_j x_j}{n_j}},\quad (x_1,\dots, x_k)\in \mathbb{Z}_{n_1} \times \cdots \times \mathbb{Z}_{n_k},
\]
for some $0 \le l_j \le n_j - 1, 1\le j \le k$. Accordingly, $\chi^{-1}$ is identified with
\[
\chi^{-1}(x_1,\dots, x_k) = \prod_{j=1}^k \exp^{\frac{-2i\pi l_j x_j}{n_j}},\quad (x_1,\dots, x_k)\in \mathbb{Z}_{n_1} \times \cdots \times \mathbb{Z}_{n_k}.
\]
\if In this paper, $\chi_l$ denotes the function
\[\chi_l:\mathbb{Z}_N\rightarrow \mathbb{C},~ \chi_l(x)=\exp^{\frac{2i \pi l x}{n}},\]
 which is an element in $\widehat{\mathbb{Z}}_N=\{\chi_l:l=0,1,2,...,N-1\}$.\fi
In particular, {$\widehat{\mathbb{Z}_2^n}$} consists of square-free monomials in $n$ variables. We have the following theorem for functions on finite abelian groups.
\begin{theorem}\cite[Chapter~1]{fulton2013representation}\label{thm:Fourier expansion}
\red{Let $G$ be a finite abelian group. Any function $f:G\to \mathbb{C}$ can be uniquely written as a linear combination of elements in $\widehat{G}$, i.e., there is a unique $ \widehat{f}: \widehat{G}\to \mathbb{C}$ such that 
\begin{equation}\label{thm:Fourier expansion:eq}
f=\sum_{\chi \in \widehat{G}} \widehat{f}(\chi)\chi.
\end{equation}
}
\end{theorem}
\red{The unique expansion of $f$ in \eqref{thm:Fourier expansion:eq} is called the \emph{Fourier expansion} of $f$. We define the \emph{support} of $f$ by $\operatorname{supp}(f)\coloneqq \{\chi: \widehat{f}(\chi)\neq 0\}$. The cardinality of  $\operatorname{supp}(f)$ is called the \emph{sparsity of $f$}.}

\subsection{Fourier sum of squares  of functions on finite abelian groups}

In this subsection, we briefly summarize the theory of Fourier sum of squares (FSOS)  developed in \cite{fawzi2016sparse,sakaue2017exact}.  The definition of FSOS is as follows:

\begin{definition}\cite{fawzi2016sparse}
Let $f$ be a nonnegative function on finite abelian group $G$, i.e $f(x)\geq 0$ for all $x \in G$, then  an FSOS representation of $f$ is in form of
\[ f=\sum_{i\in I} |g_i|^2, \]
where $\{g_i\}_{i \in I}$ is a finite family of functions on $G$.
\end{definition}

Clearly, an FSOS representation of $f$   provides a certificate of nonnegativity of $f$, making it of significantly valuable in both mathematics and computer science. The close relationship between the  FSOS representation and semidefinite programming problem is highlighted in \cite{fawzi2016sparse}. The following theorem, stated in \cite{fawzi2016sparse}, elucidates this connection.
\begin{proposition}\cite[Proposition 1]{fawzi2016sparse}\label{prop:Gram}
Let $f$ be a real-valued function on finite abelian group $G$, then $f$ has an FSOS representation if and only if there exists a Hermitian positive semidefinite matrix $Q\in \mathbb{C}^{|G|\times |G|}$, with rows and columns indexed by $\widehat{G}$, such that
\begin{equation}\label{eq:Gram:conditions}
 \widehat{f}(\chi)=\sum_{\chi=\chi'^{-1}\chi''}Q(\chi',\chi''),~\forall \chi \in \widehat{G},
\end{equation}
where $Q(\chi',\chi'')$ is the element of $Q$ with rows and columns indexed by $\chi'$ and $\chi''$. A  Hermitian positive semidefinite matrix  $Q\in \mathbb{C}^{|G|\times |G|}$ that satisfies the aforementioned conditions is called a \emph{Gram matrix} of $f$.
\end{proposition}

Moreover, for any finite abelian group $G$ and any nonnegative function $f$ on $G$,  \cite{fawzi2016sparse} directly provides the following Gram matrix:

\begin{proposition}\cite[Proposition 3]{fawzi2016sparse}\label{prop:priori-Gram-matrix}
Let $f$ be a nonnegative function on finite abelian group $G$, define the Hermitian matrix $Q\in \mathbb{C}^{|G|\times |G|}$,  $Q(\chi,\chi')=\frac{1}{|G|}\widehat{f}(\chi^{-1}\chi')$, then $Q$ is a Gram matrix of $f$.
\end{proposition}

 For ease of reference, we record below two known results about upper bound of the FSOS sparsity of nonnegative functions on special groups.
\begin{theorem}\cite[Theorem 3]{fawzi2016sparse}\label{thm:Parrilo}
Let $N,d$ be positive integers such that $d$ divides $N$. Then there exists ${T} \subseteq \widehat{\mathbb{Z}}_N$ with $\lvert {T}\rvert \le 3 d \log_2(N/d)$ such that any nonnegative function on $\mathbb{Z}_N$ of degree at most $d$ has an FSOS with support ${T}$.
\end{theorem}

\begin{theorem}\cite[Theorem 3.2]{sakaue2017exact}\label{theorem32}
Every degree $r$  nonnegative polynomial on $\mathbb{Z}_2^n$ has an FSOS certificate of degree $\lceil (n+r-1)/2 \rceil$.
\end{theorem}

\section{The problem of computing sparse FSOS and its convex relaxation}\label{sec:algorithms}\red{We formulate the problem of computing the sparsest FSOS as the optimization problem \eqref{Min:L0-1} and present its convex relaxation \eqref{1NormMin}. We prove in Theorem \ref{lem:square root} that the point-wise square root of the given function provides a solution to \eqref{1NormMin}, which can be computed in quasi-linear time in $|G|$.}
\subsection{\red{The FSOS sparsity minimization problem}}\label{SelectGramSection}
We recall that Proposition~\ref{prop:priori-Gram-matrix} gives a Gram matrix in form of
\begin{equation}\label{computeQ}
Q({\chi,\chi'})=\frac{1}{|G|}\widehat{f}(\chi^{-1}\chi'),\quad \chi,\chi'\in \widehat{G}.
\end{equation}
It is very efficient to compute $Q$ via (\ref{computeQ}),
but one usually gets a dense Gram matrix, as the next example illustrates.

\begin{example}\label{example:Z6}
The following function is considered in \cite[Example 4]{fawzi2016sparse}:
\[
f: \mathbb{Z}_6 \to \mathbb{C},\quad
f(x)=1-\frac{1}{2}(\chi(x)+\chi^{-1}(x)),\]
where $\chi(x)=\exp^{\frac{2i \pi x}{6}}$. The Gram matrix of $f$ found in \cite{fawzi2016sparse} is:
\[
  \arraycolsep=2.5pt \frac{1}{6}\cdot
\kbordermatrix{
&\chi^0  & \chi^1 & \chi^2 & \chi^3 & \chi^4 & \chi^5\\
\chi^0 &1&-1/2&0&0&0&-1/2&\\
\chi^1 &-1/2&1&-1/2&0&0&0&\\
\chi^2 &0&-1/2&1&-1/2&0&0&\\
\chi^3 &0&0&-1/2&1&-1/2&0&\\
\chi^4 &0&0&0&-1/2&1&-1/2&\\
\chi^5 &-1/2&0&0&0&-1/2&1& }.\]
By this Gram matrix and techniques from graph theory, one can only get an FSOS with sparsity at least four. However, it turns out that $f$ has another sparse Gram matrix:
\[
  \arraycolsep=2.5pt\frac{1}{6}\cdot
\kbordermatrix{
&\chi^0  & \chi^1 & \chi^2 & \chi^3 & \chi^4 & \chi^5\\
\chi^0 &3&-3&0&0&0&0&\\
\chi^1 &-3&3&0&0&0&0&\\
\chi^2 &0&0&0&0&0&0&\\
\chi^3 &0&0&0&0&0&0&\\
\chi^4 &0&0&0&0&0&0&\\
\chi^5 &0&0&0&0&0&0& },\]
from which one can obtain an FSOS of $f$ with sparsity two: $f(x)=\frac{1}{2} \left| 1-\chi(x) \right|^2$.
\end{example}

\red{By Proposition~\ref{prop:Gram}, there is a 1-1 correspondence between (sparse) FSOS of $f$ and (sparse) Gram matrices. Therefore, the problem of computing sparse FSOS can be formulated as the following optimization problem:}
\begin{equation}\label{Min:L0GramMatrix}
 \min_{Q \text{ is a Gram matrix}}  \# \left( \{\chi: \exists \chi',Q(\chi,\chi')\neq 0 \}\cup\{\chi': \exists \chi,Q(\chi,\chi')\neq 0 \} \right).
\end{equation}
Here, we label columns and rows of $Q$ by characters of the group and $Q(\chi,\chi')$ denotes the element of $Q$ labelled by $\chi$ and $\chi'$. \red{We denote by $\# S$ the cardinality of a set $S$}.

\begin{lemma}
Let  $\operatorname{diag}(Q)$ be the  vector consisting of diagonal elements of $Q$. Problem (\ref{Min:L0GramMatrix}) is equivalent to the problem of  minimizing the $\ell^0$-norm of $\operatorname{diag}(Q)$, i.e.
  \begin{equation}\label{Min:L0Lemma}
 \min_{Q \text{: Gram matrix}} \|\operatorname{diag}(Q)\|_0.
\end{equation}
\end{lemma}

\begin{proof}
As  $Q=Q^*$, we have  $Q(\chi,\chi') \neq 0$ if and only if  $Q(\chi',\chi)\neq 0$. This implies
\[\# \{\chi': \exists \chi,Q(\chi,\chi')\neq 0 \}=\# \{\chi: \exists \chi',Q(\chi,\chi')\neq 0 \}.\]
Moreover,  $Q$ is positive semidefinite, hence for $\chi,\chi' \in \widehat{G}$, $Q(\chi,\chi')\neq 0$ implies $Q(\chi,\chi)> 0$ and $Q(\chi',\chi')> 0$. Therefore, we have
 \[\#\{\chi: \exists \chi',Q(\chi',\chi)\neq 0 \}=\#\{\chi: Q(\chi,\chi)\neq 0 \},\]
which implies the equivalence between \eqref{Min:L0GramMatrix} and \eqref{Min:L0Lemma}.
\end{proof}
\subsection{\red{A convex relaxation}}Although the problem of minimizing the
$\ell^0$-norm is notoriously difficult, one can relax the problem by replacing the $\ell^0$-norm by the $\ell^1$-norm, which is a widely used method in practice \cite{candes2006stable,1564423}. Our method of solving \eqref{Min:L0Lemma} is motivated by this popular method. However, simply replacing $\|\operatorname{diag}(Q)\|_0$ by $\|\operatorname{diag}(Q)\|_1$ in \eqref{Min:L0Lemma} makes no sense in our situation. Let $f=\sum_{\chi \in \widehat{G}} \widehat{f}(\chi)\chi$ be a nonnegative function on a finite abelian group $G$, let $Q$ be a Gram matrix of $f$. Then, for any $\chi' \in \widehat{G}$, we have
\[
\sum_{\chi \in \widehat{G}}Q(\chi,\chi\chi') =  \widehat{f}(\chi').
\]
In particular, we have
\[
\sum_{\chi \in \widehat{G}}Q(\chi,\chi) = \sum_{\chi \in \widehat{G}}Q(\chi,\chi \chi_0) =  \widehat{f}(\chi_0).
\]
Here \red{$\chi_0$ is the identity element in $\widehat{G}$, i.e.,  $\chi_0(x)=1$ for all $x \in G$.} As $Q$ is positive semidefinite, we have
\begin{equation}\label{1norm}
\|\operatorname{diag}(Q)\|_1 =  \sum_{\chi \in \widehat{G}}Q(\chi,\chi) =   \widehat{f}(\chi_0).
\end{equation}
The above calculation indicates the following
\begin{lemma}\label{lem:constant l1-norm}
Given a finite abelian group $G$ and a nonnegative function $f$ on $G$, $\|\operatorname{diag}(Q)\|_1$ is a constant value for any Gram matrix of $f$.
\end{lemma}

As a consequence of Lemma~\ref{lem:constant l1-norm}, if we replace $\|\operatorname{diag}(Q)\|_0$ by $\|\operatorname{diag}(Q)\|_1$ directly in \eqref{Min:L0Lemma}, then any Gram matrix of $f$ serves as an optimizer of the relaxed problem. Therefore, a more delicate consideration is necessary to alleviate \eqref{Min:L0Lemma}. To that end, we observe that if $Q$ is a Gram matrix of $f$, then for any fixed character $\chi \in \widehat{G}$, the matrix $Q_1$ defined by
\[
Q_1(\chi',\chi'') \coloneqq Q(\chi\chi',\chi\chi''),\quad \chi',\chi''\in \widehat{G},
\]
is also a Gram matrix of $f$. Hence it is sufficient to find a sparse Gram matrix $Q$ such that $Q(\chi_0,\chi_0)\neq 0$ and the minimum of $\|\operatorname{diag}(Q)\|_0$ in \eqref{Min:L0Lemma} remains unchanged if we impose this extra condition. Thus we obtain
\begin{equation}\label{Min:L0-1}
 \min_{\substack{  Q:~\text{Gram matrix}   }} \|\operatorname{diag}(Q)\|_0=\min_{\substack{  Q:~\text{Gram matrix}  }} \#\{\chi\neq \chi_0:Q(\chi,\chi)\neq 0\}+1,
\end{equation}
which is equivalent to \eqref{Min:L0Lemma} in the sense that they have the same optimal value. The convex relaxation problem for the right-hand side optimization problem in \eqref{Min:L0-1} is
\[
\min_{  Q:~\text{Gram matrix}  } \sum_{\chi\neq \chi_0} |Q(\chi,\chi)|,
\]
which can be formulated as the following  semidefinite programming problem:
\begin{eqnarray}\label{1NormMin}
\begin{aligned}
&\min_{Q \in \mathbb{C}^{\widehat{G} \times \widehat{G}}}   &&\langle Q, A\rangle \\
&~\text{s.t.}   &&  \langle Q, B_{\chi}\rangle=\widehat{f}(\chi), \quad \chi \in \widehat{G},\\
&              && Q \succeq 0,
\end{aligned}
\end{eqnarray}
where
\[
A(\chi,\chi') = \begin{cases}
1,~\text{if}~\chi = \chi' \ne \chi_0 \\
0,~\text{otherwise}.
\end{cases},\quad
B_{\chi}(\chi',\chi'') = \begin{cases}
1,~\text{if}~ \chi'^{-1}\chi''=\chi\\%\chi' = \chi, \chi'' \ne \chi\\
0,~\text{otherwise}.
\end{cases}
\]
It is clear that
 \[\langle Q, A\rangle=\sum_{\chi \neq \chi_0}Q(\chi,\chi). \]
Since $Q \succeq 0$, we have $Q(\chi,\chi)\geq 0$ for each $\chi \in \widehat{G}$ and this implies
\[\sum_{\chi \neq \chi_0}Q(\chi,\chi)=\sum_{\chi \neq \chi_0}|Q(\chi,\chi)|.\]
The conditions $ \langle Q, B_{\chi}\rangle=\widehat{f}(\chi)$ is equivalent to  \eqref{eq:Gram:conditions}, which ensures that $Q$ is indeed a Gram matrix of $f$.
Now we arrive at the following result about the relation between \eqref{1NormMin} and  \eqref{Min:L0-1}.
\begin{proposition}\label{prop:convex relax}
The minimization problem~\eqref{1NormMin} is a convex relaxation of \eqref{Min:L0-1}. Moreover, \eqref{1NormMin} is equivalent to
\begin{eqnarray}\label{1NormMax}
\begin{aligned}
&\max_{Q \in \mathbb{C}^{\widehat{G} \times \widehat{G}}}   && Q(\chi_0,\chi_0) \\
&~\text{s.t.}   &&  \langle Q, B_{\chi}\rangle=\widehat{f}(\chi), \quad \forall \chi \in \widehat{G},\\
&              && Q \succeq 0,
\end{aligned}
\end{eqnarray}
Here $\chi_0$ is the identity element in $\widehat{G}$.
\end{proposition}
\begin{proof}
The first part of the proposition can be verified easily by a direct computation and the second claim follows immediately by recalling the fact that $\sum_{\chi \in \widehat{G}}Q(\chi,\chi)$ is equal to the constant $\widehat{f}(\chi_0)$.
\end{proof}

\subsection{\red{A closed form solution to the convex relaxation}}\label{subsec:closed-form-solution}
By the simple fact that a nonnegative function on a finite abelian group has a square root, we are able to prove that solving \eqref{1NormMin} is equivalent to computing the square root of $f$.
\begin{theorem}\label{lem:square root}
Let $f$ be a nonnegative, nonzero function on a finite abelian group $G$ and let $h$ be its square root defined by
\[
h(x)=\sqrt{f(x)},\quad x \in G.
\]
Suppose that $h =\sum_{\chi \in \widehat{G}} a_{\chi}\chi$ is the Fourier expansion of $h$ and $Q$ is the matrix defined by
\begin{equation}\label{Q0matrix}
Q_0(\chi,\chi') = \overline{a_{\chi}} a_{\chi'}.
\end{equation}
Then $Q_0$ is a  solution of   (\ref{1NormMin}).
   \end{theorem}

   \begin{proof}
First we prove  that $Q_0$ is a feasible solution of \eqref{1NormMin}. % here we prove that $Q$ is a Gram matrix of $f$.
By definition, we observe that $Q_0\succeq 0$ and $\rank(Q_0) = 1$. Since for each $x \in G$,
    \[f(x)=\left(\sum_{\chi \in \widehat{G}}a_{\chi}\chi(x)\right)\overline{\left(\sum_{\chi \in \widehat{G}}a_{\chi}\chi(x)\right)}= \sum_{\chi,\chi' \in \widehat{G}} Q_0(\chi,\chi') \chi^{-1}(x)\chi'(x),\]
$Q_0$ is a Gram matrix of $f$ of rank $1$.

Next, we prove that $Q_0$ is an optimal solution of \eqref{1NormMin}. To achieve this, we recall from Proposition~\ref{prop:convex relax} that \eqref{1NormMin} is equivalent to \eqref{1NormMax}. Thus it suffices to prove that $Q_0$  an optimal solution of \eqref{1NormMax}. 

\noindent
{\it Claim:
we claim that if $\widetilde{Q}$ is a feasible solution of \eqref{1NormMax} of rank bigger than one, then \[\widetilde{Q}(\chi_0,\chi_0) \leq {Q_0}(\chi_0,\chi_0).\] }

%We claim that if $\widetilde{Q}$ is a feasible solution of \eqref{1NormMax} of rank bigger than one, then $\widetilde{Q}(\chi_0,\chi_0) \leq {Q_0}(\chi_0,\chi_0)$. 

We assume for a moment that the claim holds true. The proof is completed by showing that $\widehat{Q}(\chi_0,\chi_0) \leq {Q_0}(\chi_0,\chi_0)$ for any rank one Gram matrix $\widehat{Q}$. To see this, we write $\widehat{Q} = u^*u$ and denote by $g$ the function corresponding to $u$. Then we have
\[
\widehat{Q}(\chi_0,\chi_0)=\frac{1}{|G|^2}\left|\sum_{x \in {G}}g(x)\right|^2 \leq\frac{1}{|G|^2} \left(\sum_{x \in {G}}|g(x)|\right)^2=\frac{1}{|G|^2}\left(\sum_{x \in {G}}h(x)\right)^2={Q_0}(\chi_0,\chi_0).
\]
Here the penultimate equality follows from $|g|^2 = f = h^2$.

Lastly it is left to prove the claim. We suppose that $\rank(\widetilde{Q}) = r \ge 2$. We recall that a spectral decomposition
    \[\widetilde{Q} =\sum_{i=1}^r u_i^*u_i\]
leads to an SOS decomposition $f=\sum_{i=1}^r \left| g_i \right|^2$, where $g_i$ is the function corresponding to $u_i$ defined by
\[
g_i=\sum_{\chi \in \widehat{G}}u_i(\chi)\chi,\quad i =1,\dots, r.
\]
In particular, we have
\begin{align*}
u_i(\chi_0) &=\langle g_i,\chi_0 \rangle=\frac{1}{|G|}\sum_{x \in {G}}g_i(x), \\
\widetilde{Q}(\chi_0,\chi_0) &=\sum_{i=1}^r |u_i(\chi_0)|^2= \frac{1}{|G|^2}  \sum_{i=1}^r \left|\sum_{x \in {G}}g_i(x)\right|^2.
\end{align*}
By the relation $h^2 = f = \sum_{i=1}^r |g_i|^2$ and Cauchy-Schwartz inequality, we obtain
\begin{align*}
\sum_{i=1}^r \left| \sum_{x \in {G}}g_i(x)\right|^2 &= \sum_{i=1}^r \sum_{x\in G} |g_i(x)|^2 +  \sum_{i =1}^r  \sum_{x \ne y} g_i(x) \overline{g_i}(y) \\
&= \left( \sum_{x\in G} h(x)^2 \right) +   \sum_{x \ne y} \sum_{i =1}^r  g_i(x) \overline{g_i}(y) \\
&\le \left( \sum_{x\in G} h(x)^2 \right) +  \sum_{x \ne y}  \sqrt{\left( \sum_{i=1}^r |g_i(x)|^2 \right) \left( \sum_{i=1}^r |g_i(y)|^2 \right)} \\
&= \left( \sum_{x\in G} h(x)^2 \right) + \sum_{x\ne y} h(x) h(y) \\
& = \left| \sum_{x\in G} h(x) \right|^2.
\end{align*}
This implies that $\widetilde{Q}(\chi_0,\chi_0) \le Q_0(\chi_0,\chi_0)$ and completes the proof of the claim.
\end{proof}

An important consequence of Theorem~\ref{lem:square root}  is that the optimization problem \eqref{1NormMin} can be  solved in $O\left(|G|\log(|G|)\right)$ time via the fast Fourier transform (FFT) and the inverse fast Fourier transform (iFFT). As an illustrative example, for xqfunction $f(x)= 1-\cos(\frac{2\pi x }{6})$ discussed in Example~\ref{example:Z6}. By the Fourier transform, we have
\begin{eqnarray*}
\nonumber  \sqrt{f}= &\left(\frac{\sqrt{2}}{3} + \frac{\sqrt{6}}{6}\right)-\left(\frac{\sqrt{2}}{12}+ \frac{\sqrt{6}}{12}\right)\left(\chi+\chi^{-1}\right)+\left(\frac{\sqrt{2}}{12} - \frac{\sqrt{6}}{12}\right)\left(\chi^2+\chi^{-2}\right)+\left(\frac{\sqrt{6}}{6} - \frac{\sqrt{2}}{3}\right)\chi^3.
\end{eqnarray*}
According to Lemma~\ref{lem:square root}, the rank one matrix $Q_0 = u^\ast u$ is a solution to \eqref{1NormMin}, where
\[
u =\begin{bmatrix}
\frac{\sqrt{2}}{3} + \frac{\sqrt{6}}{6} &
-\frac{\sqrt{2}}{12}- \frac{\sqrt{6}}{12} &
\frac{\sqrt{2}}{12} - \frac{\sqrt{6}}{12} &
\frac{\sqrt{6}}{6} - \frac{\sqrt{2}}{3} &
\frac{\sqrt{2}}{12} - \frac{\sqrt{6}}{12} &
-\frac{\sqrt{2}}{12}- \frac{\sqrt{6}}{12}
\end{bmatrix}.
\]
However, it is obvious that the matrix $Q_0$ is not sparse and actually this is also the case in general.
\subsection{\red{Square-root-based basis selection method for sparse FSOS}}\label{subsec:basis selection} \red{Although Theorem~\ref{lem:square root} already provides a solution $Q_0$ of \eqref{1NormMin}, it is not sparse.} To obtain a sparse solution of \eqref{1NormMin}, we use the magnitude of the diagonal elements of $Q_0$ as a heuristic guidance: \emph{if the term  $Q_0(\chi,\chi)$ is small, then we search for a  Gram matrix $\widetilde{Q}$ such that $\widetilde{Q}(\chi,\chi) = 0$}. Our heuristic guidance is inspired by the idea of truncation, which is widely used in scientific computing~\cite{EY1936,Elliott1965,BS2011} and numerical analysis~\cite{Atkinson2008,SB2013}. \red{However, the sparse matrix $\widetilde{Q}$ we search for is an exact Gram matrix of $f$ without error, it is indeed a sparse Gram matrix. %However, the sparse matrix $\widetilde{Q}$ we search for is not an approximation of a Gram matrix of $f$, it is indeed a sparse Gram matrix. %Roughly speaking, to find a desired sparse $\widetilde{Q}$, we want to squeeze those small diagonal elements of $Q_0$ into the submatrix with large diagonal elements. 
  We remark that although there is no guarantee on the existence of $\widetilde{Q}$, Theorem \ref{thm:square-root_based_error} supplies a rationale for this heuristic guidance. Moreover, our numerical experiments in Section~\ref{subsec:numerics} illustrate both the effectiveness and efficiency of the heuristic guidance.
  }

Suppose $Q$ is a solution of \eqref{1NormMin}. We observe that there exists a permutation $\sigma\in \mathfrak{S}_{|G|}$ such that
\begin{equation}\label{eq:ordering}
Q(\chi_{\sigma(1)},\chi_{\sigma(1)}) \geq Q(\chi_{\sigma(2)},\chi_{\sigma(2)}) \geq Q(\chi_{\sigma(3)},\chi_{\sigma(3)}) \geq \cdots \geq Q(\chi_{\sigma(|G|)},\chi_{\sigma(|G|)}).
\end{equation}
Without loss of generality, we may simply assume $\chi_k=\chi_{\sigma(k)}$ in the sequel.

For each $|G| \ge k\ge 1$, we define
 \begin{equation}\label{SatasifyMin}
 \begin{split}
   \mathcal{S}_k =\{ Q \in \mathbb{C}^{\widehat{G} \times \widehat{G}}:~& Q \text{ is Gram matrix of } f, \text{ with } Q(\chi_i,\chi_i)=0  \text{ for all } k \le i \leq |G|\}.
\end{split}
\end{equation}
By definition, each $Q \in \mathcal{S}_k$ is a Gram matrix of sparsity at most $(k-1)$. %Moreover, $\mathcal{S}_k$ is nonempty if and only if \eqref{1NormMin} has a feasible solution.
 According to the heuristic guidance, we choose $k$ such that $\mathcal{S}_k$ is nonempty via binary search. %
%pick  a small integer $t>0$, if $\mathcal{S}_t$ is empty, then we increase $t$ and repeat the procedure. Otherwise, we use binary search to find the minimal $t \ge k \ge 1$ such that $\mathcal{S}_k$ is not empty. 
\red{The problem of checking the non-emptiness of the convex set $\mathcal{S}_k$ can be solved via an SDP solver in polynomial time in $k$.
%  We choose a small number $\varepsilon>0$ and let $t$ be the minimal index such that
%\begin{equation}\label{eq:guidance}
%\frac{Q(\chi_t,\chi_t)}{Q(\chi_{0},\chi_{0})} > \varepsilon,
%\end{equation}
%where $\chi_{0}$ is the identity element of $\widehat{G}$.

We end this subsection by the following proposition on the invariance of the square-root-based basis selection method under a group isomorphism.

 \begin{proposition}\label{prop:invariant-isomorphism}
Let $G_1,G_2$ be isomorphic finite abelian groups and let $\phi:G_2\to G_1$ be an isomorphism. For any nonnegative function $f$ on $G_1$ and $\chi \in \widehat{G_1}$, we have 
\[\chi \circ \phi \in \widehat{G_2}, ~~ \widehat{\sqrt{f}}(\chi)=\widehat{\sqrt{f\circ \phi}}(\chi \circ \phi).\]
Moreover, $f$ and $f\circ \phi$  share the same  matrix $Q_0$ given in (\ref{Q0matrix}). 
\end{proposition}}
\begin{proof}
It is clear that $\chi \circ \phi \in \widehat{G_2}$, and
\[\widehat{\sqrt{f\circ \phi}}(\chi \circ \phi)=\frac{1}{|G|}\sum_{x \in G_2}\sqrt{f\circ \phi(x)} \cdot \overline{\chi \circ \phi(x)}=\frac{1}{|G|}\sum_{x \in G_1}\sqrt{f(x)}\cdot \overline{\chi(x) }=\widehat{\sqrt{f}}(\chi).\]
The moreover part follows from the fact that $Q_0$ is determined by the square root of the function.
\end{proof}

As a comparison, we consider the following example. \if For any finite abelian group $G$, the  group algebra $\mathbb{C}[\widehat{G}]$ is just the set consisting of all of the complex functions on $G$.\fi For any positive integer $n$, the group isomorphism   $\mathbb{Z}_{2\cdot 3^n}\simeq \mathbb{Z}_2 \times\mathbb{Z}_{3^n}$ leads to the following isomorphism of rings:
 \[\mathbb{C}[x,y]/\langle 1-x^{3^n},1-y^2 \rangle \cong  \mathbb{C}[\widehat{\mathbb{Z}_2 \times \mathbb{Z}_{3^n} } ] \cong \mathbb{C}[\widehat{\mathbb{Z}_{2\cdot 3^n}}]\cong \mathbb{C}[x]/\langle 1-x^{2\cdot 3^n} \rangle.   \]
 This isomorphism indicates that from the perspective of polynomial optimization,  checking the nonnegativity of  function 
 \[
 f(x)=1-\frac{1}{2}(\chi(x)+\chi^{-1}(x)), ~x \in \mathbb{Z}_{2\cdot 3^n},
 \] 
 where $\chi(x)=\exp^{\frac{x}{3^n} \pi i}$  is equivalent to checking the nonnegativity of the optimal value of either one of the following two polynomial optimization problems:
 \begin{eqnarray*}
 \nonumber
\begin{aligned}
&\min_{x \in \mathbb{C}}   && 1-\frac{1}{2}x-\frac{1}{2}x^{2\cdot 3^n-1}\\
&~\text{s.t.}  && x^{2\cdot 3^n}=1
\end{aligned}
\end{eqnarray*}
 and
  \begin{eqnarray*}
 \nonumber
\begin{aligned}
&\min_{x,y \in \mathbb{C}}   && 1-\frac{1}{2}xy^{\frac{3^n+1}{2}}-\frac{1}{2}xy^{\frac{3^n-1}{2}} \\
&~\text{s.t.}  && x^2=y^{3^n}=1.
\end{aligned}
\end{eqnarray*}
Although these two problems are essentially same,  the bases obtained by degree-based selection method are different. In fact, for any $n$, the objective function of the former problem always admits a degree-one FSOS 
\[1-\frac{1}{2}x-\frac{1}{2}x^{2\cdot 3^n-1}=\frac{1}{2}\left|1-x \right|^2=\frac{1}{2}\left(1-x\right)^*\left(1-x\right),\]
where $x^*=x^{-1}=x^{2\cdot 3^n-1}$, whereas the objective function of the latter problem has no degree-one FSOS if $n>1$.
%  the characters of degree at most $1$ of $\mathbb{Z}_6$  are $\{\chi^{-1},\chi^0,\chi^1\}$, where $\chi(x)=\exp^{\frac{2i \pi x}{6}}$, however, the characters of degree at most $1$ of $\mathbb{Z}_3 \times \mathbb{Z}_2$  are $\{\chi'^{-1},\chi',\chi_0,\chi''\}$, where $\chi'(x,y)=\exp^{\frac{2i \pi x}{3}}$, $\chi''(x,y)=\exp^{\frac{2i \pi y}{2}}$ for all $(x,y)\in \mathbb{Z}_3 \times \mathbb{Z}_2$. This example shows that  degree-based basis selection methods are not invariable under the isomorphism.

\section{An algorithm for sparse FSOS and numerical experiments}\label{subsec:numerics}
In this section, we present Algorithm \ref{alg:4} to compute a sparse FSOS of a given nonnegative function on a finite abelian group,  and then we test Algorithm~\ref{alg:4} by different classes of numerical experiments. These experiments are done in Matlab R2016b with CVX package \cite{cvx,gb08} and SDPT3 solver \cite{SDPT3}  on a desktop computer with Intel Core i9-10900X CPU (3.7GHz). Codes for these examples can be found in \url{https://github.com/jty-AMSS/FSOS}. Our numerical examples indicate the following four features of Algorithm~\ref{alg:4}:
\begin{enumerate}[(i)]
\item For nonnegative functions with a given degree bound, the cardinality of sparse FSOS is much smaller than the theoretical upper bound for FSOS sparsity given in Theorem~\ref{thm:Parrilo}.
\item Algorithm~\ref{alg:4} can deal with nonnegative functions whose minima are close to zero.
\item Algorithm~\ref{alg:4} has quasi-linear complexity in the cardinality of the group. In particular, it can deal with nonnegative functions with sparse FSOS on groups of order up to $10^7$.
\item Algorithm~\ref{alg:4} works for arbitrary finite abelian groups. Due to the page limit, we only present examples for $\mathbb{Z}_N$ and $\mathbb{Z}_N \times \mathbb{Z}_N$, but interested readers can find codes and examples for $\mathbb{Z}_2^N$ on \url{https://github.com/jty-AMSS/FSOS}.
\end{enumerate}

%\subsection{algorithm}\label{subsec:algorithm}
With all the preparations above, we are ready to present Algorithm~\ref{alg:4} which computes a sparse FSOS of a nonnegative function on a finite abelian group.
\begin{algorithm}[!ht]
\caption{sparse FSOS of a nonnegative function on a finite abelian group}
\label{alg:4}
\begin{algorithmic}[1]
\renewcommand{\algorithmicrequire}{\textbf{Input}}
\Require
nonnegative function $f$, finite abelian group $G$, a small positive number $\delta$
\renewcommand{\algorithmicensure}{\textbf{Output}}
\Ensure
sparse FSOS certificate of $f$ on $G$.
\State compute the Fourier coefficient of square root of $f$, i.e. compute $\{a_{\chi}\}_{\chi \in \widehat{G}}$ such that $\sqrt{f}=\sum_{\chi \in \widehat{G}}a_{\chi}\chi$. \label{alg4:step1}
\State Sort  $\{a_{\chi}\}$ in descending order of their absolute values and get $|a_{\chi_1}| \geq |a_{\chi_2}|\geq  \cdots \geq |a_{\chi_{|G|}}|$.\label{alg4:step2}
\State  set $k\coloneqq  \lceil \sqrt{s} \rceil $, where $s$ is the sparsity of $f$.\label{alg4:step3}
\State  Test whether $S_k$ in \eqref{SatasifyMin} is empty for $f+\delta$. If it is empty, then let $k\coloneqq 2k$ and recalculate $S_k$ until it is not empty.\label{alg4:step4}
\State By the bisection method to find  a minimum $ k_{\min} \in [1/2  k, k] $  such that $S_{k_{\min}}$ is not empty, select  $Q \in S_{k_{\min}}$.\label{alg4:step5}
\State Compute Cholesky factorization $Q=H^{*}H$, with the columns of $H$ indexed by $\{\chi_j\}_{j=1}^{k_{\min}}$.\label{alg4:step6}
\\
\Return $f=\sum_{i=1}^r |\sum_{j=1}^{k_{\min}} H(i,\chi_j)\chi_j|^2$. \label{alg4:step12}
\end{algorithmic}
\end{algorithm}
Below we supply some details of this algorithm. 

In step~\ref{alg4:step1}, we  first perform an inverse Fast Fourier transform (iFFT) on $f$, then compute the pointwise square root, and finally apply the Fast Fourier transform (FFT) to obtain the Fourier coefficients of $\sqrt{f}$, which can be done in $\operatorname{O}(|G|\log(|G|))$ time.

In step~\ref{alg4:step2}, we sort the   absolute values of Fourier coefficients instead of the  diagonal elements of $Q_0$ since $Q_0(\chi,\chi)=|a_{\chi}|^2$ holds for all $\chi \in \widehat{G}$.  This step can  be completed in $\operatorname{O}(|G|\log(|G|))$ time.

In step~\ref{alg4:step3}, since  $f$ has no FSOS with sparsity less then $\sqrt{s}$,  we set the initial value of $k$ equal to  $\lceil \sqrt{s} \rceil$. 

%As discussed in section~\ref{subsec:closed-form-solution}, 

In step~\ref{alg4:step4} and step~\ref{alg4:step5},  checking  whether $S_k$ is empty and the selection of $Q$ can be done by any  SDP solver. These two steps involve at most $2\log k_{\min}$  SDP computations, and each SDP problem has a size of at most $2k_{\min}$. In  Theorem~\ref{thm:square-root_based_error}, we will  show  that it is reasonable to add a   small perturbation $\delta$  to $f$ to obtain a sparse FSOS. 

%%In practice, due to the numerical error, we do not recommend directly checking if $S_k$ is an empty set, but rather suggest solving %the following semidefinite programming problem instead:}
%\begin{eqnarray}\label{eq:opt-M}
%\begin{aligned}
%&\min_{Q \in \mathbb{C}^{\widehat{G} \times \widehat{G}},~ M \in \mathbb{R}}   && M  \\
%&~\text{s.t.}   &&  Q \text{ is Gram matrix of } f+M,\\
%&              && \text{ Nonzero elements of }Q\text{ only appear in rows and columns indexed by elements in }  S_k.
%\end{aligned}
%\end{eqnarray} 
%\red{Then clearly, for a given  positive integer $k$,  $M\leq 0$ is equivalent to $S_k \neq \emptyset $. Therefore, in step~\ref{alg4:step4} and step~\ref{alg4:step5}, we first solve problem~\eqref{eq:opt-M} and then determine whether $M\leq \varepsilon $, where $\varepsilon \geq 0$  is determined by the required precision.  }

  As a consequence, we obtain the following proposition about the complexity of Algorithm~\ref{alg:4}.

\begin{proposition}\label{prop:complexity}
The total complexity of  Algorithm~\ref{alg:4} is at most 
\[\operatorname{O}\left(|G| \log(|G|)+\log(k_{\min})\operatorname{SDP}(2k_{\min})\right).\]
Here $\operatorname{SDP}(k)$ denotes the complexity of solving the SDP problem of size $k \times k$. % and with  $\operatorname{O}(k^2)$ constraints.
It is known that $\operatorname{SDP}(k)$ is at  most $\operatorname{O}(k^6)$ \cite{zhang2021sparse}.
\end{proposition}
\subsection{Experiments on $\mathbb{Z}_N$}
We present numerical results on randomly generated functions on groups of the form $\mathbb{Z}_N$. 
\subsubsection{The first experiment on $\mathbb{Z}_N$}
In the first experiment, we randomly pick $10$ functions $h_i,i=1,\dots, 10$, satisfying
\begin{enumerate}[(i)]
\item $\operatorname{deg}(h_i) \leq 24$.
\item $0<\min_{x \in G}h_i(x)<1$.
\item $\operatorname{Re}(\widehat{h_i}(\chi)),\operatorname{Im}(\widehat{h_i}(\chi)) \in [-10,10]$ for all $\chi \in \widehat{\mathbb{Z}_N}$.
\end{enumerate}
For each $1 \le i \le 10$, we apply Algorithm~\ref{alg:4} to $h_i$ separately. We record results in Table~\ref{tab:experiment1} in which the second column shows the average cardinality of the sparsity of FSOS found by Algorithm~\ref{alg:4} for $h_i$, the third column is the average time cost of Algorithm~\ref{alg:4} and the last column is the theoretical upper bound $3 d \log_2(N/d)$ for the FSOS sparsity given by Theorem~\ref{thm:Parrilo}.
\begin{table}[htbp]
\begin{center}
 \begin{tabular}{p{35pt}p{70pt}p{50pt}p{50pt}}
\toprule   group  & \if mean \fi FSOS sparsity&\if mean \fi time(s)&   bounds \cite{fawzi2016sparse}\\
\midrule
$\mathbb{Z}_{10000}$ & 16.7&   1.49& 648 \\
$\mathbb{Z}_{20000}$ &  18.6& 2.42 &720 \\
$\mathbb{Z}_{30000}$ &19 &  2.82 &  792 \\
$\mathbb{Z}_{40000}$ &17.8&  3.03  & 792 \\
$\mathbb{Z}_{50000}$ & 18.8 &3.38  & 864 \\
$\mathbb{Z}_{60000}$ & 19 &  3.89& 864 \\
\bottomrule
\end{tabular}
\end{center}
 \caption{First experiment: bounded degree and bounded minimum}\label{tab:experiment1}\label{table4.1}
\end{table}
Since the theoretical bound given by Theorem~\ref{thm:Parrilo} (last column) is for arbitrary nonnegative functions, it is much larger than our computed FSOS sparsity (second column).
 %while Algorithm~\ref{alg:4} is designed for sparse nonnegative functions.
\subsubsection{The second experiment on $\mathbb{Z}_N$}
In the second experiment, %\cyan{for a given group $\mathbb{Z}_N$}, 
we randomly pick $10$ functions $h_i,i=1,\dots, 10$ satisfying
\begin{enumerate}[(i)]
\item $|\operatorname{supp}(h_i)|\leq 25$.
\item $0<\min_{x\in G}(h_i(x))<1$.
\item $\operatorname{Re}(\widehat{h_i}(\chi)),\operatorname{Im}(\widehat{h_i}(\chi)) \in [-1,1]$ for all $\chi\neq \chi_0$.
\end{enumerate}
For $1 \le i \le 10$, we execute Algorithm~\ref{alg:4} separately for each $h_i$ and record the sparsity and  running time, the results are summarized in Table~\ref{tab:experiment2}. Here the second column shows the average sparsity of  FSOS certificates found by Algorithm~\ref{alg:4}, the third column is the average time cost of Algorithm~\ref{alg:4}. Notice that in the first experiment, we impose a degree bound to compare the computed FSOS sparsity and its theoretical upper bound. In this experiment, we  impose a bound on the cardinality of the support of $h_i$ to test the performance of Algorithm~\ref{alg:4} on nonnegative functions whose minimum values are  close to zero.

\begin{table}[htbp]
\begin{center}
 \begin{tabular}{p{30pt}p{70pt}p{60pt}}
\toprule   group &FSOS sparsity & time(s) \\
\midrule
$\mathbb{Z}_{10000}$ & 110.4 & 224.3 \\
$\mathbb{Z}_{20000}$ & 144.4 & 1074.6\\
$\mathbb{Z}_{30000}$ & 168.2 & 2273.7\\
$\mathbb{Z}_{40000}$ & 204.8 & 3869.0\\
$\mathbb{Z}_{50000}$ & 195.8 & 3314.3\\
$\mathbb{Z}_{60000}$ & 219.7 & 4594.6\\
\bottomrule
\end{tabular}
\end{center}
 \caption{Second experiment: bounded support and bounded minimum}\label{tab:experiment2} \label{table4.2}
\end{table}

\subsection{Experiment on $\mathbb{Z}_N \times \mathbb{Z}_N$}
We carry out an experiment on groups of the form $\mathbb{Z}_{N} \times \mathbb{Z}_{N}$. 

In this experiment, we randomly choose a subset $T \subseteq \widehat{\mathbb{Z}_{N} \times \mathbb{Z}_{N}}$ with $|T| = 10$ and randomly choose $10$ real-valued functions $g_j$, $j=1,\dots, 10$ on $\mathbb{Z}_{N} \times \mathbb{Z}_{N}$ satisfying:
\begin{enumerate}[(i)]
\item $\operatorname{supp}(g_j)\subseteq T$.
\item $\operatorname{Re}(\widehat{g_j}(\chi)),\operatorname{Im}(\widehat{g_j}(\chi)) \in [-10,10]$ for each $\chi \in T$.
\end{enumerate}
We apply Algorithm~\ref{alg:4} to find a sparse FSOS certificate of $f=\sum_{j=1}^{10}|g_j|^2$. For each value of $N$, we repeat the experiment $10$ times, and record the sparsity and the running time. In Table~\ref{tab:NxN-fsos7}, we record numerical results for different values of $N$. The second column shows the mean sparsity of FSOS certificates found by Algorithm~\ref{alg:4}. The third column is the mean time cost for each example.  In Figure~\ref{fig:NxN-fsos7}, we plot the running time cost of Algorithm~\ref{alg:4} versus the group size $N^2$. It is notable that:
\begin{enumerate}[(i)]
\item The FSOS certificate found by  Algorithm~\ref{alg:4} always has cardinality around $50$. This is because the FSOS sparsity of $f$ is at most $10$, which can be seen from its construction.
\item Algorithm~\ref{alg:4} can work for groups of  size up to $2.5 \cdot 10^7$.
\item Figure~\ref{fig:NxN-fsos7} roughly fits $N\log(N)$ since in these examples, the factor $k$ appeared in Proposition~\ref{prop:complexity} can be regarded as a constant.
\end{enumerate}
\begin{table}[htbp]
\begin{center}
 \begin{tabular}{p{70pt}p{70pt}p{50pt}}
\toprule   group  &FSOS sparsity&  time(s) \\
\midrule
$ \mathbb{Z}_{500}\times \mathbb{Z}_{500}$ & 51.4	&24.3 \\
$ \mathbb{Z}_{1000}\times \mathbb{Z}_{1000}$ & 50.8	&63.3\\
$ \mathbb{Z}_{1500}\times \mathbb{Z}_{1500}$ & 49.2	&123.3\\
$ \mathbb{Z}_{2000}\times \mathbb{Z}_{2000}$ & 49.6	&208.4\\
$ \mathbb{Z}_{2500}\times \mathbb{Z}_{2500}$ & 50.6	&318.6\\
$ \mathbb{Z}_{3000}\times \mathbb{Z}_{3000}$ & 50.2	&457.6\\
$ \mathbb{Z}_{3500}\times \mathbb{Z}_{3500}$ & 49.2	&632.8\\
$ \mathbb{Z}_{4000}\times \mathbb{Z}_{4000}$ & 49.8	&831.7\\
$ \mathbb{Z}_{4500}\times \mathbb{Z}_{4500}$ & 48.2	&1066.0\\
$ \mathbb{Z}_{5000}\times \mathbb{Z}_{5000}$ & 50.6	&1325.1\\
\bottomrule
\end{tabular}
\end{center}
 \caption{Third experiment: bounded FSOS support}\label{tab:NxN-fsos7}\label{table4.3}
\end{table}
\begin{figure}[!ht]
\centering
\begin{tikzpicture}
\begin{axis}[
    xlabel=Cardinality of group $(N^2)$,
    scaled ticks=false,
    xtick={0,10000000,20000000},
    xticklabel style={/pgf/number format/precision=8},
    ylabel=Mean value of running time(s),
    legend style={at={(0.5,-0.2)},anchor=north}
    ]
%\addplot[sharp plot,mark=*,blue] plot coordinates {
\addplot[sharp plot,blue] plot coordinates {
(250000 , 24.27591207)
(1000000 , 63.26254083)
(2250000 , 123.3243038)
(4000000 , 208.4085652)
(6250000 , 318.5596792)
(9000000 , 457.5723177)
(12250000 , 632.8103091)
(16000000 , 831.5679976)
(20250000 , 1066.024842)
(25000000 , 1325.139295)
};
\addlegendentry{Group form: $\mathbb{Z}_N\times \mathbb{Z}_N$}
\end{axis}
\end{tikzpicture}
\caption{Time complexity}
\label{fig:NxN-fsos7}
\end{figure}
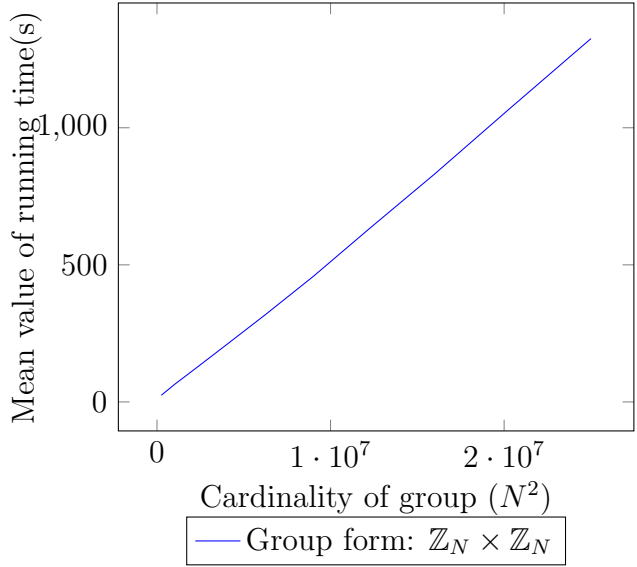

\section{\red{Bounds on FSOS sparsities}}\label{sec:square-root:bound}
\red{In this section we focus on bounding the FSOS sparsity of a nonnegative function. In Subsection~\ref{subsec:dominating constant term}, we prove that for functions with dominating constant terms, their FSOS sparsities are bounded by their Fourier sparsities. In Subsection~\ref{subsec:error bound}, we prove that a suitable perturbation of the given function admits an FSOS supported on the Fourier support of the original function. As a consequence, we obtain an error analysis of step~\ref{alg4:step4} in Algorithm~\ref{alg:4} which provides a rationale for the square-root-based basis selection method we proposed in Subsection~\ref{subsec:basis selection}.}
\if \cyan{by giving a perturbed FSOS with  a given set $S$ (or a given sparsity $s$).}\fi 
%giving a feasible solution of problem~\ref{eq:opt-M} with a given set $S$ (or a given sparsity $s$). }
\subsection{FSOS sparsities of functions with dominating constant terms}\label{subsec:dominating constant term} We first bound the FSOS sparsity of a nonnegative function with dominating constant term.
\begin{proposition}\label{Newcoro310}
Let $G$ be a finite abelian group and let $S$ be a symmetric subset of $\widehat{G}$, $f=\sum_{\chi\in S}\widehat{f}(\chi)\chi$ be a real-valued function on $G$ such that $ \widehat{f}(\chi_0)  \geq \sum_{\chi\neq \chi_0}|\widehat{f}(\chi)|$, then the FSOS sparsity of $f$ is at most $|S|$.
\end{proposition}

\begin{proof}
It is sufficient to prove that $f$ has a Gram matrix whose number of nonzero rows is at most $|S|$. Let $S=\{\chi_0,\chi_1,...,\chi_k\}$. Here $\chi_0$ is  the trivial character. We define $Q\in \mathbb{C}^{\widehat{G}\times \widehat{G}}$ by:
\[
Q(\chi, \chi') = \begin{cases}
|G|\left(a_{\chi_0}-\frac{1}{2}\sum_{\chi \neq \chi_0}|\widehat{f}(\chi)|\right),~\text{if}~\chi = \chi'  = \chi_0,\\
\frac{|G|}{2}\widehat{f}(\chi'), ~\text{if}~\chi = \chi_0 \ne \chi', \\
\frac{|G|}{2} \widehat{f}(\chi^{-1}), ~\text{if}~\chi' = \chi_0 \ne \chi, \\
\frac{|G|}{2}|\widehat{f}(\chi)|,~\text{if}~\chi = \chi' \ne \chi_0, \\
0,~\text{otherwise}.
\end{cases}
\]
Obviously nonzero elements of $Q$ are on rows and columns labelled by characters in $S$. Moreover, $Q$ is an arrowhead matrix, i.e., nonzero elements of $Q$ either lie in the diagonal or in the first row or first column.

It is left to prove $Q\succeq 0$. It is sufficient to verify the nonnegativity of each principal minor of $Q$. Since $Q$ is an arrowhead matrix, a principal submatrix of $Q$ not involving elements in the first row  and column is a diagonal matrix. Thus it is straightforward to verify that the determinant of such a principal submatrix is nonnegative.

If a principal submatrix $Q$ does include the first row and column, then we may compute its determinant directly. For instance, the determinant of $Q(S,S)$ is
  \[
  \operatorname{det}(Q(S,S))=\left(Q(\chi_0,\chi_0)-\sum_{\chi \neq \chi_0 \in S}\frac{Q(\chi_0,\chi){Q(\chi,\chi_0)}}{Q(\chi,\chi)}\right)\prod_{\chi \neq \chi_0 \in S}Q(\chi,\chi),
  \]
where $Q(S,S)$ denotes the principal submatrix of $Q$ obtained by selecting rows and columns labelled by characters in $S$. For each $\chi \in \widehat{G}$ we have
  \[Q(\chi_0,\chi)=\frac{|G|}{2} \widehat{f}(\chi)=\frac{|G|}{2}\overline{\widehat{f}(\chi^{-1})}= \overline{Q(\chi,\chi_0)}.\]

Since $f$ has a dominating constant term, we have
  \[Q(\chi_0,\chi_0)-\sum_{\chi \neq \chi_0 \in S}\frac{Q(\chi_0,\chi){Q(\chi,\chi_0)}}{Q(\chi,\chi)}=|G|\left(\widehat{f}(\chi_0)- \sum_{\chi \neq \chi_0}|\widehat{f}(\chi)| \right) \geq 0.\]
The nonnegativity of other principal minors of $Q$ can be proved similarly and this completes the proof.
\end{proof}

\begin{remark}
We recall that Theorem~\ref{thm:Parrilo} supplies an upper bound $3 d \log_2(N/d)$ for a nonnegative function $f$ on $\mathbb{Z}_N$, where $d$ is the  degree of $f$. But according to Proposition~\ref{Newcoro310}, a degree $d$ nonnegative function with dominating $\widehat{f}(\chi)$ on $\mathbb{Z}_N$ has FSOS sparsity at most $(2d+1)$, which is much smaller than $3 d \log_2(N/d)$ for large $N$.
\end{remark}

\subsection{FSOS support of a perturbation}\label{subsec:error bound}
Let $G$ be a finite abelian group and let $f$ be a function on $G$. We  define the $\ell^1$ norm of $\widehat{f}$ as
\begin{equation}\label{eq:def:L1-norm}
  \|\widehat{f}\|_{\ell^1}\coloneqq\sum_{\chi \in \widehat{G}} |\widehat{f}(\chi)|.
\end{equation}

\red{We begin with the following lemma, which can be regard as an improvement of Lemma~1 in \cite{bach2022exponential} with  finite abelian group constraints.}

\begin{lemma}\label{lem:detailed:L1-norm}
Let $G$ be a finite abelian group and let $f$ be a real-valued function on $G$. If $S$ is a subset of $\widehat{G}$ such that
\[\operatorname{supp}(f)\subseteq S\cdot S^{-1}\coloneqq \{\chi'\chi^{-1}:\chi,\chi' \in S\} ,\]
then $ f+\|\widehat{f}\|_{\ell^1}$ admits  an FSOS with support $S$.
\end{lemma}
\begin{proof}
It is clear that  $S\cdot S^{-1}$  is a symmetric set.  %i.e. $\chi \in S$ implies $\chi^{-1} \in S$. 
Let $S_0$ be the set of symmetric elements in $S\cdot S^{-1}$, 
\[S_0\coloneqq \{\chi \in S\cdot S^{-1}:\chi^{-1}=\chi\},\] then we can decompose the set  $S\cdot S^{-1} \backslash S_0\coloneqq \{\chi \in S\cdot S^{-1}: \chi \notin S_0\}$ into two disjoint  sets: 
\[S\cdot S^{-1} \backslash S_0=S_{1}\cup S_{-1}, ~~ S_{-1}=S_1 ^{-1}= \{\chi^{-1}:\chi \in S_1\},\]
and there are no mutually inverse elements in the set $S_1$. Since $f$ is a real-valued function, $f=\overline{f}$, and $\widehat{f}(\chi^{-1})=\overline{\widehat{f}(\chi)}$ holds for all $\chi \in \widehat{G}$, then $f$ can be expressed in the following form:
\[f=\sum_{\chi \in S_1}\widehat{f}(\chi)\chi+\sum_{\chi \in S_1}\overline{\widehat{f}(\chi)}\chi^{-1}+\sum_{\chi \in S_0}{\widehat{f}(\chi)}\chi,\]
and 
\[\|\widehat{f}\|_{\ell^1}=2\sum_{\chi \in S_1}|\widehat{f}(\chi)|+\sum_{\chi \in S_0}|\overline{\widehat{f}(\chi)}|.\]

Since $S_1$ and $S_0$ are subsets of $ S\cdot S^{-1}$, we have that for any $\chi \in S_0\cup S_1$, there exists $\chi',\chi'' \in S$ such that $\chi =\chi'' \chi'^{-1}$. Then we have  
\[\|f\|_{\ell^1}+f=\sum_{\chi \in S_1}\left( 2|\widehat{f}(\chi)|+\widehat{f}(\chi)\chi+\overline{\widehat{f}(\chi)}\chi^{-1} \right)+\sum_{\chi \in S_0}\left( |\widehat{f}(\chi)|+\widehat{f}(\chi)\chi \right). \]

%For the two terms in the right of above  expression, we discuss them as follows:
We show below that $\|f\|_{\ell^1}+f$  has an FSOS with support $S$:

\begin{itemize}
  \item For  ${\chi=\chi''\cdot \chi'^{-1} \in S_1}$, $ 2|\widehat{f}(\chi)|+\widehat{f}(\chi)\chi+\overline{\widehat{f}(\chi)}\chi^{-1} ={|\widehat{f}(\chi)|} \left| \chi'+{\frac{\widehat{f}(\chi)}{|\widehat{f}(\chi)|}} \chi''  \right|^2$.
  \item For ${\chi=\chi'' \cdot \chi'^{-1}\in S_0}$, since  $\overline{f}=f$, $\widehat{f}(\chi)=\overline{\widehat{f}(\chi)}\in \mathbb{R}$,  we have \[|\widehat{f}(\chi)|+\widehat{f}(\chi)\chi=\frac{|\widehat{f}(\chi)|}{2} \left|1+\chi\right|^2=\frac{|\widehat{f}(\chi)|}{2} \left|\chi''+\frac{\widehat{f}(\chi)}{|\widehat{f}(\chi)|}\chi'\right|^2.\]
\end{itemize}
%In conclusion, we have provided a FSOS of 
\end{proof}

Based on Lemma \ref{lem:detailed:L1-norm}, we  present the following theorem to show that a perturbation of $f$ has an FSOS supported in a given support.
\begin{theorem}\label{thm:square-root_based_error}
Let $G$ be a finite abelian group and let $f\neq 0$ be a nonnegative function on $G$. Assume that $S$ is a subset of $\widehat{G}$ such that $\operatorname{supp}(f)\subseteq S$ and $S=S^{-1}$. We define $h$ by
\begin{equation}\label{hfun}
\widehat{h}(\chi)=\begin{cases}
                                       \widehat{\sqrt{f}}(\chi), & \mbox{if } \chi \in S,\\
                                       0, & \mbox{if } \chi \notin S. \\
                                     \end{cases}
                                     \end{equation}
            
Then $f+M$ has an FSOS supported in $S$ for 
\begin{equation} \label{Mvalue}
M\coloneqq  2\|\widehat{\sqrt{f}}-\widehat{h}\|_{\ell^1} \cdot\|\widehat{h}\|_{\ell^1}+\|\widehat{\sqrt{f}}-\widehat{h}\|_{\ell^1}^2.
\end{equation}
\end{theorem}
\red{Before we proceed to the proof of the theorem, we remark that Theorem \ref{thm:square-root_based_error} 
actually provides a rationale of the square-root-based basis selection method proposed in Subsection \ref{subsec:basis selection}.
Indeed, we can regard the function $h$ in Theorem~\ref{thm:square-root_based_error} as the truncation of $\sqrt{f}$ by $S$.  
Suppose $S$ consists of characters with large coefficients in the Fourier expansion of $\sqrt{f}$. Then $\sqrt{f}-h$ only contains terms of $\sqrt{f}$ with small coefficients. Thus $\|\widehat{\sqrt{f}}-\widehat{h}\|_{\ell^1}$ and $M$ defined in \eqref{Mvalue} are small. As a consequence, Theorem \ref{thm:square-root_based_error} implies that a small perturbation $f+M$ admits an FSOS supported in $S$.}
%We can regard the function $h$ in Theorem~\ref{thm:square-root_based_error} as restricting  $\sqrt{f}$  to the set $S$.  
%If the distance $\|\widehat{\sqrt{f}}-\widehat{h}\|_{\ell^1}$
%is small, then $M$ is small, and a small perturbation of $f$ can be  written   as an FSOS with support $S$. 
%Theorem \ref{thm:square-root_based_error}  confirms the rationality of the selection strategy discussed in Section \ref{subsec:closed-form-solution}, which chooses the parts of $\sqrt{f}$ with larger coefficients. There the value   $\|\widehat{\sqrt{f}}-\widehat{h}\|_{\ell^1}$ is usually very small. 

%The essence of Theorem \ref{thm:square-root_based_error} can be roughly described as that  the error of the FSOS of $f$ with support $S$ is bounded by the ``weight" of $S$ of $\sqrt{f}$. Theorem \ref{thm:square-root_based_error} also confirms the rationality of the selection strategy discussed in Section \ref{subsec:closed-form-solution}, which chooses the parts of $\sqrt{f}$ with larger coefficients, thereby reducing $\|\widehat{\sqrt{f}}-\widehat{h}\|_{\ell^1}$. 
\begin{proof}
Define $g \coloneqq f-|h|^2 $, since $\operatorname{supp}(f) \subseteq S$,  $f\neq 0$, $\widehat{f}(\chi_0)=\frac{1}{|G|}\sum_{x \in G}f(x)>0$. Hence we have 
\[\chi_0 \in S, ~~\operatorname{supp}(f) \subseteq S \subseteq S \cdot S.\]
Moreover, we have 
\[\operatorname{supp}(h^2)  \subseteq   S\cdot S\coloneqq\{\chi\cdot \chi':\chi,\chi' \in S\}=S\cdot S^{-1}.\]

Since $S^{-1}=S$ and $\sqrt{f}$ is real-valued, we have $\overline{\widehat{h}(\chi)}={\widehat{h}(\chi^{-1})}$ holds for all $\chi \in \widehat{G}$, thus $h$ is also a real-valued function and $|h|^2=h^2$. We denote $S'\coloneqq \widehat{G}\setminus S$. According to (\ref{hfun}), we have 
\[g:=f-|h|^2=\left(\sum_{\chi' \in S'} \widehat{\sqrt{f}}(\chi')\chi'\right)^2+2\cdot \left(\sum_{\chi \in S } \sum_{\chi' \in S'} \widehat{\sqrt{f}}(\chi)\widehat{\sqrt{f}}(\chi')\chi\chi'\right).\]
We have
\begin{equation}\label{gfunnew}
\|\widehat{g}\|_{\ell^1}\leq \left(\sum_{\chi' \in S'}  \left|\widehat{\sqrt{f}}(\chi')\right|\right)^2+2\left(\sum_{\chi \in S } \sum_{\chi' \in S'}\left| \widehat{\sqrt{f}}(\chi)\widehat{\sqrt{f}}(\chi')\right|\right).
\end{equation}
Let
\begin{equation}\label{Mper}
M\coloneqq  2\|\widehat{\sqrt{f}}-\widehat{h}\|_{\ell^1} \cdot\|\widehat{h}\|_{\ell^1}+\|\widehat{\sqrt{f}}-\widehat{h}\|_{\ell^1}^2=2\left(\sum_{\chi \in S } \sum_{\chi' \in S'}\left| \widehat{\sqrt{f}}(\chi)\widehat{\sqrt{f}}(\chi')\right|\right)+\left(\sum_{\chi' \in S'}  \left|\widehat{\sqrt{f}}(\chi')\right|\right)^2.
\end{equation}
By (\ref{gfunnew}) and (\ref{Mper}),  we have 
%It is easy to check that 
%\[ \|\widehat{\sqrt{f}}-\widehat{h}\|_{\ell^1}=\sum_{\chi \in S'}\left|\widehat{\sqrt{f}}(\chi')\right|,\]
%and
%\[\|\widehat{\sqrt{f}}-\widehat{h}\|_{\ell^1} \cdot\|\widehat{h}\|_{\ell^1}=\sum_{\chi \in S } \sum_{\chi' \in S'}\left| \widehat{h}(\chi)\widehat{\sqrt{f}}(\chi')\right|.\]
\[ \|\widehat{g}\|_{\ell^1} \leq M.\]
By Lemma \ref{lem:detailed:L1-norm}, $M+g$ has an FSOS with support $S$. Hence $f+M=|h|^2+(M+g)$ has an FSOS with support $S$. 

\end{proof}

 %Additionally, by selecting the largest coefficient of $\sqrt{f}$, we can obtain the following estimation.

The following corollary shows that for a given $s$ and function $f$, an appropriate perturbation of $f$ admits an FSOS of sparsity at most $s$.

\begin{corollary}\label{coro:speed-of-FSOS}
\red{Let $G$ be a finite abelian group and let $f \ne 0 $ be a nonnegative function on $G$. Then for any positive integer $s >1+|\operatorname{supp}(f)|$, $f+M_{s'} $ admits an FSOS with sparsity at most $s$, where $s'=s-1-|\operatorname{supp}(f)| $ and 
\[M_{s'}=\|\widehat{\sqrt{f}}\|_{\ell^1}^2\left(3-5\frac{s'}{|G|}+2\frac{s'^2}{|G|^2}\right).
\] 
}
\end{corollary}
\begin{proof}
%Since $0 \leq f \leq 1$, we have $ 0 \leq \widehat{f}(\chi_0)=\frac{1}{|G|}\sum_{x \in G}f(x) \leq 1$. 
Since $\sqrt{f}$ is nonnegative, $|\widehat{\sqrt{f}}(\chi)|=|\widehat{\sqrt{f}}(\chi^{-1})|$ for all $\chi \in \widehat{G}$. Without loss of generality, we can arrange $\widehat{G}$ in descending order of their absolute values in $\widehat{\sqrt{f}}$, with any two mutually inverse elements always adjacent  i.e. $|\widehat{\sqrt{f}}(\chi_1)|\geq |\widehat{\sqrt{f}}(\chi_2)|\geq \cdots \geq |\widehat{\sqrt{f}}(\chi_{|G|})|$, and for any integer $i>0$, $\chi_i^{-1}$ must be one of the three characters $\chi_{i-1}$, $\chi_{i}$ or $\chi_{i+1}$. Then for any integer $k>0$, the sum of the first $k$ largest coefficients satisfies
\[\sum_{i=1}^{k}|\widehat{\sqrt{f}}(\chi_i)|\geq k\frac{\|\widehat{\sqrt{f}}\|_{\ell^1} }{|G|}.\]
In this case, either the set $\{\chi_1,\chi_2,\cdots \chi_s\}$ is symmetric, or the set $\{\chi_1,\chi_2,\cdots \chi_s,\chi_{s+1}\}$ is symmetric. Let $S'$ be the symmetric set among these two sets, $S=\operatorname{supp}(f)\cup S'$, and let $h$ be the truncation of $\sqrt{f}$ at set $S$ defined in (\ref{hfun}). Then we have 
\[\|\widehat{\sqrt{f}}-\widehat{h}\|_{\ell^1}\leq \|\widehat{\sqrt{f}}\|_{\ell^1}-s'\frac{\|\widehat{\sqrt{f}}\|_{\ell^1} }{|G|}.\]
According to (\ref{Mper}), we have 
\[M \leq \|\widehat{\sqrt{f}}\|_{\ell^1}^2-s'\frac{\|\widehat{\sqrt{f}}\|_{\ell^1}^2 }{|G|}+2\left(\|\widehat{\sqrt{f}}\|_{\ell^1}-s'\frac{\|\widehat{\sqrt{f}}\|_{\ell^1} }{|G|}\right)^2=\|\widehat{\sqrt{f}}\|_{\ell^1}^2\left(3-5\frac{s'}{|G|}+2\frac{s'^2}{|G|^2} \right)=M_{s'}.\]
By Theorem~\ref{thm:square-root_based_error}. we can conclude  that  $f+M_{s'} $ admits an FSOS with sparsity at most $s$.
\end{proof}

We can make an estimate of $\|\widehat{\sqrt{f}}\|_{\ell^1}$ for $0 \leq f \leq  1$.  The Fourier coefficient $\widehat{f}(\chi_0)$ equals to $\sum_ {\chi \in \widehat{G}}|\widehat{\sqrt{f}}(\chi)|^2$  as
\[ \widehat{f}(\chi_0)=\frac{1}{|G|}\sum_{x \in G}f(x).\]
Hence, we have $0 \leq \sum_ {\chi \in \widehat{G}}|\widehat{\sqrt{f}}(\chi)|^2\leq 1$ and $\|\widehat{\sqrt{f}}\|_{\ell^1}=\sum_ {\chi \in \widehat{G}}|\widehat{\sqrt{f}}(\chi)|\leq \sqrt{|G|}$.

We remark that Theorem \ref{thm:square-root_based_error} and corollary~\ref{coro:speed-of-FSOS} only depend on the coefficients  of  $\sqrt{f}$ and the cardinality of $G$, regardless of the degree of $f$ and the structure of $G$. Therefore, this result can be regarded as a complement to results in \cite{slot2023sum} for functions with high degree.

\section{Applications of FSOS}\label{sec6}
We conclude this paper by a discussion on applications of FSOS in combinatorial optimization problems and sum of Hermitian squares of polynomials on tori.
\subsection{Combinatorial  optimization}
Combinatorial optimization  is a very natural resource of applications of FSOS \cite{CC2012,CKP2000,CK2005,Lecoutre2013,PR2002,Tsang2014}. 
\subsubsection{\red{Certificate problem for MAX-SAT}}
It has been  proved in \cite{yang2022short} that FSOS supply short  certificates for MAX-SAT, MIN-SAT and UNSAT problems.
In addition, MAX-2SAT and MAX-3SAT problems can be solved by optimizing polynomials on $\mathbb{Z}_2^n = \{-1,1\}^n$. In order to reduce the size of the related SDP problems, some choices of monomial bases are proposed~\cite{VVH2008}. These monomial bases perform well on some benchmark problems but poorly on others. As an example, we consider the weighted MAX-2SAT problem corresponding to the function $g: \mathbb{Z}_2^{10} \to \mathbb{Z}$:
\begin{eqnarray*}
g= &&50450 + 234x_3 - 1386x_2 - 1389x_1 + 502x_4 + 3056x_5 - 4692x_6 - 2142x_7 - 1312x_8 \\
&&- 4645x_9 + 3787x_{10}   - 3399x_1x_2 - 1140x_1x_3  - 282x_2x_3 - 2413x_1x_5 - 884x_2x_4 \\
&&- 2212x_1x_6 + 3457x_2x_5
+ 4462x_3x_4 - 2002x_5x_{10}+  2057x_3x_9  + 4097x_1x_7 + 1707x_2x_6  \\
&&+ 3419x_1x_8 - 4102x_2x_7 - 976x_3x_6 - 2403x_4x_5  - 1245x_1x_9  - 3786x_2x_8 - 1122x_6x_7   \\
&&+ 1014x_3x_7 + 3139x_4x_6 + 483x_1x_{10} + 4417x_2x_9 - 854x_3x_8  - 2037x_5x_6   
- 1678x_2x_{10}\\
&& + 667x_6x_8   
 - 491x_1x_4   - 981x_4x_8 + 4848x_5x_7 + 4085x_3x_{10} + 1129x_4x_9 - 4936x_5x_8 \\
&& - 2628x_4x_{10} + 2787x_5x_9  - 936x_3x_5 + 640x_6x_9 + 1874x_7x_8  - 707x_6x_{10}  + 778x_7x_9 \\
&&+ 3813x_7x_{10}  - 2764x_8x_9 + 3038x_8x_{10} + 2170x_9x_{10}   + 6x_4x_7.
\end{eqnarray*}
The monomial basis suggested by \cite{VVH2008} consists of all monomials of degree at most $2$, i.e.
\[M_{ap}=\{1\} \cup \{x_i:i=1,2,\cdots,10\} \cup \{x_ix_j:1\leq i < j \leq 10\}. \]
It can be checked that there exists no SOS supported in $M_{ap}$. However, our algorithm produces an FSOS of $g$ of sparsity $52< | M_{ap}|=56$. More examples can be found in \cite{yang2023fourier,yang2023lower}.
\subsubsection{The pigeon-hole principle}
In the following, we consider a more remarkable application in combinatorial optimization: the proof complexity of the pigeon-hole principle. First of all, we recall that the pigeon-hole principle says that $n + 1$
pigeons cannot be placed into $n$ holes unless a hole contains more than one pigeon. For each positive integer $n$, we define a conjunctive normal form (CNF) formula in $(n+1)n$ variables $\{p_{ij}\}_{i=1,j=1}^{n+1,n}$:
\[
\overline{\operatorname{PHP}_n^{n+1}} \coloneqq \bigwedge_{i=1}^{n+1}\left({p_{i1}\lor p_{i2} \lor \cdots \lor p_{in} }\right)\land  \bigwedge_{1 \leq i<k \leq n+1, 1\leq j \leq n}
\left(\lnot p_{ij}\lor \lnot p_{kj}\right).
\]
Then the pigeon-hole principle is equivalent to the statement that $\overline{\operatorname{PHP}_n^{n+1}}$ is unsatisfiable for all $n\in \mathbb{N}$. The standard technique to prove the unsatisfiability of a CNF formula is the resolution refutation\cite{davis1994computability,bonet2007resolution}. According to the theorem that follows, proving the pigeon-hole principle by resolution refutation is difficult.
\begin{theorem}\cite[Theorem 16, Corollary 18]{bonet2007resolution}
For sufficiently large $n$, any resolution refutation of $\overline{\operatorname{PHP}_n^{n+1}}$ requires $2^{n/20}$ inference steps.
\end{theorem}

It turns out that, however, we are able to prove the pigeon-hole principle by an FSOS  certificate of sparsity $O(n^2)$. To this end, we define:
\[p_n:\mathbb{Z}_n^{n+1} \mapsto \mathbb{C},~ p_n(x_1,x_2,\ldots,x_{n+1})=\sum_{1 \leq i<j\leq n+1}\operatorname{Eqv}(x_i,x_j),\]
where
\[\operatorname{Eqv}:\mathbb{Z}_n^2 \mapsto \mathbb{C},~\operatorname{Eqv}(x,y)=\begin{cases}
                                                                     1, & \mbox{if } x=y, \\
                                                                                  0, & \mbox{otherwise}.
\end{cases} \]

\begin{proposition}\label{prop6.2}
For each $n\in \mathbb{N}$, the unsatisfiability of $\overline{\operatorname{PHP}_n^{n+1}}$ is equivalent to the positivity of $p_n$. Moreover, $p_n$ admits an FSOS of sparsity at most $O(n^2)$.
\end{proposition}
\begin{proof}
Let $\chi_k(x)=\exp(\frac{2\pi i kx}{n})$ and let $\operatorname{NOR}$ be the function defined on $\mathbb{Z}_n$:
\[\operatorname{NOR}(x)=\begin{cases}
                                                                     1, & \mbox{if } x=0, \\
                                                                                  0, & \mbox{otherwise}.
\end{cases}
\]
%$\operatorname{NOR}(0)=1$ and $\operatorname{NOR}(x)=0$  for all $0 \neq x \in \mathbb{Z}_n$.
By the inner product $\langle\chi_k, \operatorname{NOR}\rangle = \sum_{x \in \mathbb{Z}_n}\operatorname{NOR}(x) \chi_k(x)=\chi_k(0)=1$, we can conclude that  $\operatorname{NOR}=\frac{1}{n}\sum_{k=1}^n \chi_k$. It is easy to verify that $\operatorname{Eqv}(x,y)=\operatorname{NOR}(x-y)=\frac{1}{n}\sum_{k=1}^n \chi_k(x)\chi_{n-k}(y)$, thus
\[p_n=\frac{n+1}{2}+\frac{1}{n}\sum_{k=1}^{n-1}\sum_{1 \leq i < j \leq n+1}\chi_k(x_i)\chi_{n-k}(x_j).\]
By the fact that
\[\left|\sum_{i=1}^{n+1}\chi_k(x_i)\right|^2=n+1+\sum_{1 \leq i \neq j \leq n+1}\chi_k(x_i)\chi_{n-k}(x_j),\]
we have
\[p_n=\frac{n+1}{2n}+\sum_{k=1}^{n-1} \frac{1}{2n}\left| \sum_{i=1}^{n+1}\chi_k(x_i)   \right|^2.\]
This implies $p_n> \frac{1}{2}>0$ for all $n\in \mathbb{N}$, with an FSOS  certificate of sparsity $O(n^2)$.
\end{proof}

\subsection{\red{Sum of Hermitian squares (SOHS) of polynomials on $\mathbb{T}^n$}}\label{sec:lift}
%In this section, we discuss a particular property of low-degree FSOS, that is, low-degree FSOS can be lift to SOHS.  Based on this property, we can also give some lower bound of degree of FSOS of some special nonngegative functions.

As a counterpart of SOS for non-negative polynomials over $\mathbb{C}$, SOHS  has also been  extensively studied in polynomial optimization and mathematical physics \cite{josz2018lasserre,klep2008sums,wang2022exploiting}.  As an interesting application of sparse FSOS, we show below that one can construct an SOHS of $f\geq 0$ on $\mathbb{T}^n$ from an FSOS of $f \circ \tau $:
\[
\underbrace{\Gamma_N \times \cdots \times \Gamma_N}_{\text{$n$ copies}} \xhookrightarrow{\tau}  \mathbb{T}^n \coloneqq \underbrace{\mathbb{S}^1 \times \cdots \times \mathbb{S}^1}_{\text{$n$ copies}}
\xrightarrow{f} \mathbb{R}_+,
\]
Here $N$ is a positive integer, $\Gamma_N = \{\exp^{\frac{2 \pi i k}{N}}\}_{k=0}^{N-1} \simeq \mathbb{Z}_N$, and $\tau$ is the natural inclusion map. It is obvious that an SOHS of $f$ provides an FSOS of $f \circ  \tau$. On the other hand,  if $N$ is chosen sufficiently large, then one can construct an SOHS of $f$ from an FSOS of $f \circ \tau $ by simply replacing $\chi_k$ by $z^k$ and  $\chi_{N-k}$ by $\overline{z}^k$ for $k<\frac{N}{2}$. 
\begin{example}
In order to  computing an SOHS of   $f(z) = 1 - \frac{1}{2}(z + \overline{z})$ on $\mathbb{S}^1$, we 
choose   $N = 6$,  and compute an FSOS of  $f\circ \tau$ on $\Gamma_6$:
 \[f\circ \tau = 1 - \frac{1}{2}\chi_1-\frac{1}{2}\chi_5 = \frac{1}{2}\left|1-\chi_1\right|^2.\]
 Replacing   $\chi_1$ by $z$, we obtain an SOHS of $f$ on $\mathbb{S}^1$: 
 \[f(z) = \frac{1}{2}(1-z)(1-\overline{z}).\]
\end{example}
%Next, we will prove the theorem of the lift of low-degree FSOS.

%Let $f \in \mathbb{C}[x_1,\overline{x_1},x_2,\overline{x_2},\cdots,x_n,\overline{x_n}]$ be a polynomial, where $\overline{x_i}$ is the complex conjugate of $x_i$, we have the following Lemma:
%The proof of the following lemma is obvious. 
%\begin{lemma}\label{Lemma:SOHS}
Let  $I$ be the ideal generated by $\overline{x_i}x_i-1$, $i=1,2,\cdots,n$. Let $\rho$ be a  natural homomorphism. \[\rho:\mathbb{C}[x_1,\overline{x_1},\cdots,x_n,\overline{x_n}] \mapsto \mathbb{C}[x_1,\overline{x_1},\cdots,x_n,\overline{x_n}]/I.\]
 For a given polynomial   $f \in \mathbb{C}[x_1,\overline{x_1},x_2,\overline{x_2},\cdots,x_n,\overline{x_n}]$, suppose we have an SOHS of $\rho(f)$
  \[\rho(f)=\sum_{i=1}^m h_i\overline{h_i},\]
for  $h_i \in \mathbb{C}[x_1,\overline{x_1},\cdots,x_n,\overline{x_n}]/I, i=1..m$,  
then 
\[f=\sum_{i=1}^m h_i\overline{h_i}+g\]
for some $g \in I$ and $h_i \in \mathbb{C}[x_1,\overline{x_1},\cdots,x_n,\overline{x_n}]$. Since  $g(x)=0$ for all $x \in \mathbb{T}^n$, we have  \[f(x)=\sum_{i=1}^m\left|h_i(x)\right|^2, \forall x \in \mathbb{T}^n. \]
 Therefore, we can prove the  nonnegativity of $f$ on $\mathbb{T}^n$ by giving a sum of Hermitian squares of $\rho(f)$. For ease of reading, we will no longer distinguish between $f$ and $\rho(f)$ in the following text. 
 %Additionally, for  representative element $f\in \mathbb{C}[x_1,\overline{x_1},\cdots,x_n,\overline{x_n}]/I$, 
 Since $x_i\overline{x_i}=1$, we require that each monomial of $f$ does not contain $x_i$ and $\overline{x_i}$ simultaneously. 
 
 For a given polynomial \[f=\sum_{\alpha,\beta \in \mathbb{N}^n}a_{\alpha,\beta} \prod_{i=1}^n x_i^{\alpha_i} \overline{x_i}^{\beta_i}\in  \mathbb{C}[x_1,\overline{x_1},\cdots,x_n,\overline{x_n}]/I,\]
 where  $a_{\alpha,\beta} \in \mathbb{C}$, we define:
\begin{itemize}
  \item  $[n]=\{0,1,2,\cdots,n-1\}$.
  \item For $i=1,2,...,n$, let   $\deg_{x_i}(f)$ equal to the maximal degree of $x_i$ and $\overline{x_i}$ in $f$.
  \item The \emph{restriction} of $f$ to  the  group $G=\mathbb{Z}_{k_1}\times \mathbb{Z}_{k_2}\times,\cdots, \mathbb{Z}_{k_n}$ is the function
   \begin{align}\label{resfG}
       f|_G:&G\mapsto \mathbb{C},\\
   f|_G(x_1,x_2,\cdots,x_n)&=\sum_{\alpha,\beta \in \mathbb{N}^n}a_{\alpha,\beta}  \prod_{i=1}^n \chi_{\alpha_i-\beta_i}(x_i)=\sum_{\alpha,\beta \in \mathbb{N}^n}a_{\alpha,\beta}  \prod_{i=1}^n \exp^{{2i\pi (\alpha_i-\beta_i) x_i}/{k_i}}.
   \end{align}
  % where $\chi_{\alpha_i-\beta_i}(x)=\exp^{{2i\pi (\alpha_i-\beta_i) x}/{k_i}}$.  
\end{itemize}

For a group function
\[g:G\mapsto \mathbb{C},~ g(x_1,x_2,\cdots,x_n)=\sum_{ \alpha \in [k_1]\times[k_2]\times\cdots \times [k_n]  }a_{\alpha} \prod_{i=1}^n\chi_{\alpha_i}(x_i),\]
 for $ \alpha \in [k_1]\times[k_2]\times\cdots \times [k_n]$,  $\chi_{\alpha}$ denotes the character $\Pi_{i=1}^n \chi_{\alpha_i}$ of the abelian group $\mathbb{Z}_{k_1}\times \mathbb{Z}_{k_2}\times,\cdots, \mathbb{Z}_{k_n}$.
We define the \emph{lift} of $g$ to be a polynomial $L(g)$ on 
$\mathbb{T}^n$:
\[\operatorname{L}(g)=\sum_{ \alpha \in [k_1]\times[k_2]\times\cdots \times [k_n]  }a_{\alpha}x_i^{\tau_{k_i}(\alpha_i)} \in \mathbb{C}[x_1,\overline{x_1},\cdots,x_n,\overline{x_n}]/I,\]
%in $\mathbb{C}[x_1,\overline{x_1},\cdots,x_n,\overline{x_n}]/I$, 
where $\tau_{k_i}(j)=j$ if $j<\frac{k_i}{2}$ and  $\tau_{k_i}(j)=j-k_i$ if $j\geq \frac{k_i}{2}$. Here $x_i^{-j}=\overline{x_i}^j$ in $\mathbb{C}[x_1,\overline{x_1},\cdots,x_n,\overline{x_n}]/I$. 
 It is clear that  the lift  of $g$  is a  linear map satisfying  $\operatorname{L}(\overline{g})=\overline{\operatorname{L}(g)}$.
%We have the following Theorem.

The following theorem gives sufficient conditions that an FSOS of $f|_G$ on $G$ can be lifted to an SOHS of $f$ on $\mathbb{T}^n$. 
\begin{theorem}\label{Thm:lift}
  Let $f \in \mathbb{C}[x_1,\overline{x_1},\cdots,x_n,\overline{x_n}]$ be a polynomial defined on $\mathbb{T}^n$, %$\mathbb{C}[x_1,\overline{x_1},\cdots,x_n,\overline{x_n}]/I$, 
  $f|_G$ be its  \emph{restriction} on  the group $G=\mathbb{Z}_{k_1}\times \mathbb{Z}_{k_2}\times,\cdots, \mathbb{Z}_{k_n}$, $S$ be a subset of $\widehat{G}$. Suppose the following conditions are satisfied: 
  \begin{enumerate}
    \item For all $i=1,2,...,n$, $k_i>4 \deg_{x_i}(f)$; \label{Left:cond-1}
    \item $f|_G=\sum_{i =1}^m |h_i|^2$ is an FSOS on the  group $G$, with $\bigcup_{i \in I} \operatorname{supp}(h_i) \subseteq S$;\label{Left:cond-2}
    \item For all $\chi_{\alpha} \in S$, and all $i=1,2,...,n$, we have $0 \leq \alpha_i<\frac{k_i}{4} $ or  $n \geq \alpha_i>\frac{3k_i}{4} $ holds.\label{Left:cond-3}
  \end{enumerate}
  Then we can lift  the polynomial  $\{h_i\}_{i =1}^m$ to give an SOHS of $f$ on $\mathbb{T}^n$.

\end{theorem}
\begin{proof}
%Let $Q \in \mathbb{C}^{S\times S}$ be the Gram matrix of FSOS $\{h_i\}$, then we proof $Q$ is also a Gram matrix of SOHS. Since $Q$ is a gram matrix, we have
%\[ \sum_{\chi_{\alpha},\chi_{\beta}}Q(\chi_{\alpha}, \chi_{\beta})=\widehat{g}(\overline{\chi_{\alpha}}\chi_{\beta}) ,~\forall \chi_{\alpha}, \chi_{\beta}\in S.  \]
%Since  $ \forall\chi_{\alpha} \in S$ salsifies $\alpha_i<\frac{n_i}{4} $ or  $\alpha_i>\frac{3n_i}{4} $, then for any $\chi_{\alpha} ,\chi_{\beta} \in S$, we have $\tau_{n_i}(\alpha_i)-\tau_{n_i}(\beta_i)$.
%Let $G=\mathbb{Z}_{k_1}\times \mathbb{Z}_{k_2}\times\cdots \mathbb{Z}_{k_n}$, $f$ in $\mathbb{C}[x_1,\overline{x_1},\cdots,x_n,\overline{x_n}]/I$, 
%Let $f|_G$ be the restriction of $f$ on group $G$, let $\operatorname{L}(g)$ denote the lift of function $g: G\mapsto \mathbb{C}$, then 
The condition \ref{Left:cond-1} ensures that $f=\operatorname{L}(f|_G)$. 

For any given functions $h,g:G \mapsto \mathbb{C}$, suppose   $\operatorname{supp}(h),\operatorname{supp}(g) \subseteq S$, then  we have
\begin{eqnarray*}
&\operatorname{L}(h\cdot g)&=\operatorname{L}\left(\sum_{\chi_{\alpha},\chi_{\beta} \in S}\widehat{h}(\alpha)\widehat{g}(\beta)\chi_{\alpha}\chi_{\beta}\right)\\
&&=\sum_{\chi_{\alpha},\chi_{\beta} \in S}\widehat{h}(\alpha)\widehat{g}(\beta)\prod_{i=1}^n x_i^{\tau(\alpha_i)+\tau(\beta_i)}.
\end{eqnarray*}
If the condition \ref{Left:cond-3} holds,  then one can verify that 
\[\tau(\alpha_i)+\tau(\beta_i)=\tau(\alpha_i+\beta_i),\] and  $\operatorname{L}(h\cdot g)=\operatorname{L}(h)\cdot\operatorname{L}(g)$.
Therefore, if all three conditions hold,  then % we have $f|_G=\sum_{i=1}^m h_i \overline{h_i}$, and the support of $h_i,\overline{h_i}$ are in $S$, then:
we have
\begin{align*}
f=\operatorname{L}(f|_G)=\operatorname{L}\left( \sum_{i=1}^m h_i \overline{h_i} \right)= \sum_{i=1}^m\operatorname{L}\left( h_i\right)\operatorname{L}\left( \overline{h_i}\right)= \sum_{i=1}^m\operatorname{L}\left( h_i\right)\overline{\operatorname{L}\left( {h_i}\right)}.
\end{align*}
\end{proof}

\begin{remark}\label{rmk6.5}
As an important application of SOHS of polynomials on $\mathbb{T}^n$, one can apply it to verify  that a polynomial $f(x_1,\ldots, x_n)$ is nonnegative over intervals:  we can first map the intervals to $[-2,2]^n$, and then   choose properly a product group  $G$ to    construct  an SOHS of $f(z_1+\overline{z_1},\ldots, z_n+\overline{z_n})$ on  $\mathbb{T}^n$.  Comparing  with   known methods (based on computing SOS over constraints),    this new  method  may provides us a more concise  certificate of nonnegativity of polynomials on  given intervals. 
\end{remark}

We show below  that a sparse  FSOS on finite abelian  groups can be used to certify the nonnegativity of  the Motzkin polynomial on $[-2,2]\times [-2,2]$.

\begin{example}\label{motzkin}
Let $M(x,y)=x^4y^2+x^2y^4-3x^2y^2+1$ be the Motzkin polynomial. It is well-known that $M$ has no SOS over reals~\cite{KLYZ09,motzkin1967arithmetic}. However,  we can compute an SOS of $M(x,y)$ subject to constraints  $4-x^2 \geq 0$ and  $4-y^2\geq 0$. The  SOS of $M$ on $[-2,2]\times [-2,2]$  computed by TSSOS~\cite{wang2021tssos}\footnote{We thank Jie Wang for his help} has sparsity $9$:
\[
M(x,y) = \sum_{j=1}^3 v_j Q_j v_j^\tp + (4-x^2) v_1 Q_4 v_1^\tp  + (4-y^2) v_1 Q_5 v_1^\tp,
\]
where $v_1 = [1, x^2, y^2], v_2 = [x, x^3, xy^2], v_3 = [y, x^2y, y^3]
$ and
\begin{align*}
Q_1 &= \begin{bmatrix}
0.114 & -0.057 &-0.057\\
 -0.057 & 0.0425 & 0.014\\
-0.057 & 0.014 & 0.043
\end{bmatrix}, Q_2 = \begin{bmatrix}
1.111 & -0.277 & -0.834 \\ -0.277 & 0.074 & 0.203 \\ -0.834  & 0.203 & 0.630
\end{bmatrix},\\
Q_3 &= \begin{bmatrix}
1.111 & -0.834  & -0.277\\ -0.834 & 0.630 & 0.203\\ -0.277 & 0.203 & 0.074
\end{bmatrix}, 
Q_4 =\begin{bmatrix}
0.111 &-0.087 &-0.024 \\
-0.087 &0.074 &0.013\\ -0.024 &0.013 &0.011
\end{bmatrix},\\
Q_5 &= \begin{bmatrix}
0.111 &-0.024 &-0.087 \\
-0.024 &0.011 &0.013\\
-0.087 &0.013 &0.074
\end{bmatrix}.
\end{align*}

Our new method works as follows: %On the other hand, it is clear that if $z\in \mathbb{S}^1$, then $z+ \overline{z} \in [-2,2]$, which makes it possible to  certify the nonnegativity of polynomials with  real coefficients over a given interval.
%For a given polynomial  $f(x)$,  we can certify the nonnegativity of $f(x)$  on $[-2,2]^n$ by properly choosing a product group and compute a sparse SOHS of $f(z+\overline{z})$ for $z \in \mathbb{T}^n$. 
%We take $M(x,y)$ as an example. 
First,  we substitute of $x,y$ by  $x=z_1+\overline{z_1}, y=z_2+ \overline{z_2}$ and consider 
\[
f(z_1, z_2) \coloneqq M(z_1 + \overline{z_1}, z_2 + \overline{z_2}), \quad z_1,z_2 \in \mathbb{S}^1.
\]
Then we take $N = 8$ and compute a sparse FSOS of the function $f\circ \tau$ on $\Gamma_{8} \times \Gamma_{8}$ by Algorithm~\ref{alg:4}, which   can be lifted  further to give an  SOHS of $f$ on $\mathbb{T}^2$:
\[
f(z_1,z_2)= \left|z_1^2z_2^2 + \overline{z_1}^2 + \overline{z_2}^2 + z_1^2\overline{z_2}^2 + 2z_1^2 + z_2^2 + 2\right|^2.
\]
In particular, we obtain a rank one SOHS of $M(z_1 + \overline{z_1}, z_2 + \overline{z_2})$ for $z_1,z_2 \in \mathbb{S}^1$ of sparsity $7$, which proves that $M(x,y)\geq 0$  on $[-2,2] \times [-2,2]$ as 
\[
M(x,y) = f( e^{i \theta}, e^{i \psi}) = \left|e^{2(\theta + \psi)i } + e^{-2\theta i}  + e^{-2\psi i}  + e^{2(\theta - \psi)i}  + 2e^{2\theta i}  + e^{2\psi i}  + 2\right|^2,
\]
where $\theta = \arccos(x/2), \psi = \arccos(y/2)$ for $x, y \in [-2,2]$. 

\end{example}

 We denote $\deg_i(\chi_{\alpha})= \min\{\alpha_i,k_i-\alpha_i\}$ as the  degree of the i-th component of  $\chi_{\alpha}$.  According to Theorem ~\ref{Thm:lift}, a low-degree FSOS can always be lifted  to an  SOHS  of $f\geq 0$ on $\mathbb{T}^n$. On the other hand, this theorem also tells us that  if   $f$ is not nonnegative on $ \mathbb{T}^n$, then the restriction of $f$ on any finite abelian  group $G$  exists no FSOS with  $\deg_i(\chi_{\alpha}) \leq \frac{k_i}{4}$  for all $i \in [n]$.

\if
Let $g$ be a function on group  $G=\mathbb{Z}_{k_1}\times \mathbb{Z}_{k_2}\times\cdots\times \mathbb{Z}_{k_n}$, $f$ is the lift of $g$ , such that:
    \begin{enumerate}
    \item For all $i=1,2,...,n$, $k_i>4 \deg_{x_i}(f)$; \label{contrapositive:cond-1}
    \item $f$ is not nonnegative on $\mathbb{T}^n$. \label{contrapositive:cond-2}
    \end{enumerate}
    Then $g$ has no FSOS with support $S=\{\chi_{\alpha}:0 \leq \alpha_i<\frac{k_i}{4}   \text{ or }  n \geq \alpha_i>\frac{3k_i}{4} \}$.
\fi
%If we regard the degree of the i-th component $\chi_{\alpha}$ as $\min\{\alpha_i,k_i-\alpha_i\}$, then the  contrapositive proposition of Theorem~\ref{Thm:lift} can be stated as: If the lift of $g$ is not nonnegative on $ \mathbb{T}^n$ , then $g$ exists no FSOS with degree of the i-th component always less than $\frac{k_i}{4}$  for all $i \in [n]$. We use the following example to demonstrate this theorem.

In \cite{blekheranm2016sums}, the authors provide nonnegative  quadratic functions on group $\mathbb{Z}_2^n$,  which admit no SOS with degree less than $\frac{n}{2}$.  We  present below a similar example. % an analogous function for a more general class of groups. 

\begin{example}
\label{coro:linear-function-with-no-low-degree-FSOS}
For any $n\in \mathbb{N}$  with $n>4$, the function
\[g: \mathbb{Z}_n \to \mathbb{C},~x \mapsto 2\cos\left(\frac{2\pi}{2n}\right)-\exp^{\frac{2\pi i}{2n}}\chi_{1}(x)-\exp^{-\frac{2\pi i}{2n}}\chi_{-1}(x) \]
 is nonnegative on $\mathbb{Z}_n$,  but has no FSOS with support
\begin{equation}\label{ex6.6}
\{ \chi_{\lfloor -\frac{n}{4} \rfloor},\chi_{\lfloor -\frac{n}{4} \rfloor+1},\ldots ,\chi_{\lfloor \frac{n}{4} \rfloor-1},\chi_{\lfloor \frac{n}{4} \rfloor}\},
\end{equation}
where $\chi_k(x)=\exp^{2i\pi k x/n}$, and $\chi_{-k}=\chi_{n-k}$. 
Otherwise, according to Theorem \ref{Thm:lift},  an FSOS of $g$ satisfying (\ref{ex6.6}) can be lifted to an SOHS of 
$ L(g)=2\cos\left(\frac{2\pi}{2n}\right)-\exp^{\frac{2\pi i}{2n}}x-\exp^{-\frac{2\pi i}{2n}}\overline{x},$  which contradict to the fact that 
\[L(g)\left(\exp^{-\frac{2\pi i}{2n}}\right)=2\cos\left(\frac{2\pi}{2n}\right)-2<0,  ~{\text{for}}~ n>4.\]

%For any integer $x \in \mathbb{Z}_n$, $g(x)=2\cos(\frac{\pi}{n}) -2\, \cos\left(\frac{2\pi x +\pi}{ n}\right)$, and the maximum  value of $\cos\left(\frac{2\pi x +\pi}{ n}\right)$ is $\cos(\frac{\pi}{n})$ since the cosine function attains its maximum value at $0$ and $2\pi$. It has no such FSOS since the lift
%\[f: \mathbb{T}\to \mathbb{R},~x \mapsto 2\cos(\frac{2\pi}{2n})-\exp^{\frac{2\pi i}{2n}}x-\exp^{-\frac{2\pi i}{2n}}\overline{x}\]
%is not nonnegative, by $f(\exp^{-\frac{2\pi i}{2n}})=2\cos(\frac{2\pi}{2n})-2<0$.
\end{example}

This example can be equivalently described as: for each integer $n$,  the polynomial optimization problem 
 \begin{eqnarray*}
 \nonumber
\begin{aligned}
&\min_{x \in \mathbb{C}}   && 2\cos\left(\frac{2\pi}{2n}\right)-\exp^{\frac{2\pi i}{2n}}x-\exp^{-\frac{2\pi i}{2n}}\overline{x}\\
&~\text{s.t.}  && x^n=1
\end{aligned}
\end{eqnarray*}
has no SOS certification of degree less than or equal to  $\frac{n}{4}$.

\nocite{}
\bibliographystyle{elsarticle-num}
\bibliography{optimization}



\end{document}